\documentclass{article}
\usepackage{amsmath,amsfonts,amsthm,amssymb,mathtools}

\author{Piotr Rudnicki\thanks{Institute of Mathematics, University of Warsaw, Poland, \tt piotr.rudnicki@mimuw.edu.pl}   ~~~and~~ Andrzej Weber\thanks{Institute of Mathematics, University of Warsaw, Poland, {\tt aweber@mimuw.edu.pl}\hfill\break A.W. supported by NCN grant 2016/23/G/ST1/04282 (Beethoven 2)}}

\title{Characteristic classes  of Borel orbits  of square-zero upper-triangular matrices}

\usepackage{tgpagella}
\usepackage[T1]{fontenc}
\usepackage[tracking]{microtype}

\usepackage[lmargin=3cm, rmargin=3cm, headsep=0.5cm, headheight=0.3cm, footskip=1cm]{geometry}

\usepackage{tipa} 
\usepackage{bbm} 

\usepackage{tikz}
\usepackage{tikz-cd}
\usetikzlibrary{patterns}

\usepackage{hyperref}
\usepackage{relsize}
\usepackage{enumerate}

\theoremstyle{plain}
\newtheorem{theorem}{Theorem}[section]
\newtheorem*{nono-theorem}{Theorem}
\newtheorem{lemma}[theorem]{Lemma}
\newtheorem{cor}[theorem]{Corollary}
\newtheorem{prop}[theorem]{Proposition}

\theoremstyle{definition}
\newtheorem{defn}[theorem]{Definition}

\theoremstyle{remark}
\newtheorem{rmk}[theorem]{Remark}
\newtheorem{ex}[theorem]{Example}

\newcommand{\setsep}{\hspace{1mm} | \hspace{1mm}}

\newcommand{\C}{\mathbb{C}}

\newcommand{\Z}{\mathbb{Z}}

\newcommand{\T}{\mathbb{T}}
\newcommand{\PP}{\mathbb{P}}
\newcommand{\GL}{\operatorname{GL}}
\newcommand{\SL}{\operatorname{SL}}
\newcommand{\B}{\operatorname{B}}

\newcommand{\Inv}{\operatorname{Inv}}
\newcommand{\Hom}{\operatorname{Hom}}
\newcommand{\mC}{\operatorname{mC}}

\newcommand{\bett}{\operatorname{\beta}}
\newcommand{\Fell}{\operatorname{\mathcal F\hskip-1pt\ell}}

\def\sk{}

\newcommand{\trans}[2]{\left(#1 , #2 \right)}

\usepackage{xcolor}
\hypersetup{
  colorlinks,
  linkcolor={red!50!black},
  citecolor={blue!50!black},
  urlcolor={blue!80!black}
}

\usepackage[all]{xy}

\def\Sch{\mathfrak{S}}
\def\Sn{\mathcal{S}_n}

\def\NN{\mathcal{N}}
\def\O{\mathcal{O}}
\def\X{\mathcal{X}}
\def\piu{{\underline\pi}}
\def\jeden{{\mathbbm{1}}}
\def\csm{c_{\text{\rm SM}}}

\def\Ell{{\mathcal E\sk \ell\ell}}

\begin{document}

\maketitle

\begin{abstract}
Anna Melnikov provided a parametrization of  Borel orbits in  the affine variety
of square-zero $n \times n$ matrices by the set of involutions in the
symmetric group. A related combinatorics leads to a construction a Bott-Samelson type resolution of the orbit closures.
This allows to compute cohomological and K-theoretic invariants of the orbits: fundamental classes, Chern-Schwartz-MacPherson classes and motivic Chern classes in torus-equivariant theories. The formulas are given in terms of Demazure-Lusztig operations.
The case of square-zero upper-triangular matrices is rich enough to include information about cohomological and K-theoretic classes of the double Borel orbits in $\Hom(\C^k,\C^m)$ for $k+m=n$. We recall the relation with double Schubert polynomials and show analogous interpretation of Rim\'anyi-Tarasov-Varchenko trigonometric weight function.
\end{abstract}

The most prominent example of singular spaces with linear group action present in literature is the case of the closures of nilpotent orbits in the Lie algebra of a reductive group. The new discoveries concerning characteristic  classes for singular varieties were not applied to   nilpotent orbits so far. The exception is the whole nilpotent cone. It has well known desingularization -- the Springer resolution.
This resolution can be identified with the cotangent bundle of the generalized flag variety.
The Springer resolution was studied by many authors in the context of characteristic classes, starting from \cite{BBM} and recently in particular  \cite{AMSS0, AMSS, LSZZ, SZZ}.
By
\cite{Fu}  any symplectic resolution of the nilpotent orbit is of the form  of the cotangent bundle $T^*G/P$, where $P$ is a parabolic subgroup. The symplectic resolution is not unique. Any such pair of resolutions  differ by a sequence of locally trivial family of special flops,
\cite{Namikawa}.
The nilpotent orbits not admitting any symplectic resolution appear in all types except $A_n$ and $G_2$. Now we will study only the case $A_n$.
\medskip

Each nilpotent $G$-orbit decomposes into Borel orbits.
It is remarkable that there is a finite number of $B$-orbits for elements of  nilpotence order (height) two.
For higher nilpotence order the number of $B$-orbits is infinite except low dimensional cases.
Bender-Perrin \cite{BenderPerrin} has described a  resolution of 2-nilpotent $B$-orbit closures.
The construction is inductive, similar to the construction of the Bott-Samelson resolution of Schubert varieties. It is a mixture of the Bott-Samelson resolution and it forms an associated vector bundle. We present that construction in a constructive way, so that it can be used to compute cohomological invariants of orbits.
\medskip

We are primarily interested in the characteristic classes.
One can think about characteristic classes as a deformation of the notion of the fundamental class, as advertised in a review article \cite{Rimanyi}. We will focus on the invariants listed below. They are of cohomological nature and live in Borel equivariant cohomology or in equivariant K-theory.
We will consider the equivariant theories for the maximal torus.
\begin{itemize}
\item The fundamental classes in equivariant cohomology and equivariant K-theory serve as a starting point, as well as an occasion to compare our approach with classical Schubert calculus.
\item The \emph{Chern-Schwartz-MacPherson} classes are the invariants of singular varieties  having good covariant functorial properties. Their existence was conjectured by Grothendieck and Verdier. The first functorial construction was given by MacPherson \cite{MacPherson}. It demands possibility of resolving singularities. Thus the theory is available for varieties in characteristic zero. In a series of papers Aluffi (e.g. \cite{AluffiHyper, Aluffi, AluffiMarcolli}) computed many examples, gave an alternative definition and connected this theory with various problems in enumerative geometry. An equivariant version was introduced by Ohmoto \cite{Ohmoto, WeHefei}.

\item The $\chi_y$-genus defined by Hirzebruch for smooth varieties admits an extension  for singular varieties and the underlying homology class was defined in \cite{BSY}. An equivariant version (especially for torus action) were defined and studied in \cite{WeSEL}. A need and profits from introducing equivariant version is clear when one deals with classical  singular varieties like Schubert varieties. By the localization theorem computation may be reduced to local calculus of power series. The K-theoretic counterpart of the Hirzebruch class seems to be more natural for problems originating from representation theory. The definition was given in \cite{BSY} and the equivariant version was introduced in \cite{FRW,AMSS}.
\end{itemize}
\medskip

The structure allowing to compute the characteristic classes of Schubert varieties in the generalized flag variety $G/B$ serves as the model situation. The equivariant cohomology and K-theory with respect to the maximal torus $\T\subset B$  admit actions of the Demazure-Lusztig-type operations. According to \cite{AMSS} the Demazure-Lusztig operations permute the motivic Chern classes of Schubert varieties.
The construction is motivated by a geometric interpretation of the Hecke algebra as presented in \cite{ChrissGinzburg}. An elliptic version of the Demazure-Lusztig operations were used in \cite{RW} to compute  the elliptic characteristic classes for Schubert varieties in $G/B$ and to identify them with stable envelopes of Okounkov \cite{AgaOko}. There are two families of operations constructed   on the elliptic cohomology of $G/B$. The operations are indexed by the simple reflections $s\in W$, the generators of the Weyl group. Let's denote the mentioned operations (for the purpose of this exposition) by $\mathcal{BS}_s$ and $\mathcal{R}_s$. The Schubert varieties are indexed by the elements of the Weyl group. According to the main result of \cite{RW} the elliptic classes $\Ell(X_w)$ satisfy two recursions
$$\mathcal {BS}_s(\Ell(X_w))=\Ell(X_{ws})\quad\text{if}\quad{\rm length}(ws)>{\rm length}(w)\,,$$
$$\mathcal R_s(\Ell(X_w))=\Ell(X_{sw})\quad\text{if}\quad{\rm length}(sw)>{\rm length}(w)\,.$$
The Bernstein-Gelfand-Gelfand-type formulas serve as a prototype for the first relation, it is called the Bott-Samelson recursion. The second one (corresponding to the left action of simple reflections) is a rewritten R-matrix relation.
The elliptic class of any Schubert variety $X_w$ can be computed inductively with respect to the length of $w$, starting from the class of $X_{id}$ which is a point. For cohomology (including quantum cohomology) and  K-theory such operations were described in \cite{MihalceaNaruse}. We find analogous relations for characteristic classes of 2-nilpotent $B$-orbits.
\medskip

The nilpotent elements of nilpotence order at most two for the group $\GL_n$  are the matrices, whose square is zero $$A^2=0\,.$$
In addition we assume that $A$ is upper-triangular.
The variety of such matrices were studied in \cite{Melnikov}. The $B$-orbits are indexed by certain combinatorics.
To see the first nontrivial example of nilpotent orbits consider the 2-nilpotent elements contained in $\NN\subset \mathfrak{gl}_3$. These are the matrices

\hfil{ $\begin{pmatrix}0~~a~~b\\[-0.05cm]0~~0~~c\\[-0.05cm]0~~0~~0\end{pmatrix}$}\ \   with\ \  $ ac=0$.

\noindent There are three 2-nilpotent upper-triangular $B$-orbits of rank one
 $$\mathcal O_1=\{a=0,~c\neq 0\}\,,\quad
\mathcal O_2=\{a\neq 0,~c=0\}\,,\quad
\mathcal O_3=\{a=0,~c= 0,~b\neq 0\}\,.$$
The characteristic classes of $\mathcal O_1$ and $\mathcal O_2$ can be obtained from the minimal one -- the fundamental class $\mathcal O_3$ -- by suitably modified Demazure operations.
In this case the closures of orbits are smooth, but starting from $n=4$ singularities appear. The closure of the $\B$ orbit of the matrix
\vskip2pt
\hfil { $\begin{pmatrix}0\ 0\ 0\ 0\\[-0.1cm]0\ 0\ 1\ 0\\[-0.1cm]0\ 0\ 0\ 0\\[-0.1cm]0\ 0\ 0\ 0\end{pmatrix}$}

\noindent
consists of the block matrices {\footnotesize $\begin{pmatrix}0\ A\\0\,\ 0\end{pmatrix}$}, where $A$ is a 2$\times$2 singular matrix.
\medskip

The Bender and Perrin resolution of 2-nilpotent $B$-orbits allows to apply the methods of \cite{RW, KRW} and to give recursive formula for characteristic classes.
Let us present the geometric picture.
All the square-zero matrices of a fixed order $r$ are conjugate. They form a $\GL_n$ orbit. Every such matrix $A$ is characterized by its kernel $K$, the image $W$ and a nondegenerate map $\C^n/K\to W$.  From this data one constructs a resolution of the closure of the $\GL_n$-orbit. Let $\Fell(r,n-r,n)$ be the 2-step flag variety  consisting of pairs of subspaces $W\subset K$ in $\C^n$ with $\dim(W)=r$, $\dim(K)=n-r$.
There are natural maps
$$
\begin{matrix}&
&\Hom(\C^n/K,W)&\stackrel\varphi\longrightarrow &\overline {\GL_n\cdot A}&\subset&
\Hom(\C^n,\C^n)\\
&&\downarrow\\
&&\Fell(r,n-r,n).\end{matrix}$$
Moreover the resolution map $\varphi$ factorizes through the bundles
$$\Hom(\C^n/W,W)\qquad\text{and}\qquad \Hom(\C^n/K,K)$$ over the corresponding Grassmannians giving rise to symplectic resolutions.
The natural action of the Borel subgroup decomposes the singular space $\overline{\GL_n\cdot A}$ as well as its resolutions.
To construct a resolution of the Borel orbit closure it is enough to resolve a suitable Schubert variety in the Grassmannian and pull back the bundle of Hom's. A precise algorithm is given by Theorem \ref{th:resolution}.
\medskip

In the present paper the Chern-Schwartz-MacPherson  classes $\csm$ and motivic Chern classes $\mC$ are computed, as well as the fundamental classes.
The ambient space here is the vector space of strictly upper-triangular matrices in $\mathfrak{gl}_n$, denoted by $\NN$.

\begin{nono-theorem}[Theorems \ref{th:fund-indukcja} and \ref{char-indukcja}] There exist operators $\bett_i$ and $A_i$ acting on the fraction field of $H^*_\T(\NN)\simeq \Z[t_1,t_2,\dots t_n]$ and $\bett_i^K$ and $A_i^K$ acting on the fraction field of $K_\T(\NN)[y]=\Z[t_1^{\pm1},t_2^{\pm1},\dots t_n^{\pm1},y]$ having the following properties.
Suppose $\O_1\subset\NN$  is an $B$-orbit of a square-zero matrix, and suppose $\O_2=Bs_i\O_1\subset\NN$, where $s_i$ is the matrix of a simple reflection.  Assume  $\dim\O_2=
\dim\O_1+1$. Then the fundamental classes, Chern-Schwartz-MacPherson classes and Motivic Chern classes can be computed by the formulae
$$\begin{matrix}
{[\overline \O_2]}={e(\NN)}\cdot\bett_i\hskip-3pt\left(\frac{[\overline \O_1]}{e(\NN)}\right)\in H^{2{\rm codim}\, \O_2}_\T(\NN)\,,&
{\csm(\O_2\subset\NN)}={e(\NN)}\cdot A_i\hskip-3pt\left(\frac{\csm(\O_1\subset\NN)}{e(\NN)}\right)\in H^*_\T(\NN)\,,\\ \\
{[\overline \O_2]_K}={e^K(\NN)}\cdot\bett^K_i\hskip-3pt\left(\frac{[\overline \O_1]_K}{e^K(\NN)}\right)\in K_\T(\NN)\,,&
{\mC(\O_2\subset \NN)}={e^K(\NN)}\cdot A^K_i\hskip-3pt\left(\frac{\mC(\O_1\subset \NN)}{e^K(\NN)}\right)\in K_\T(\NN)[y]\,.
\end{matrix}$$
\end{nono-theorem}
Here $e(\NN)$ and $e^K(\NN)$ stand for the equivariant Euler classes in cohomology and K-theory. The operators appearing in the theorem are versions of Hecke operators from \cite{AMSS0} or \cite{AMSS}.
The notion of the K-theoretic class in not
ambiguous  in this case, since the singularities of  the square-zero orbit closures are rational, see Remark \ref{non_ambiguity}.

\medskip
At the first sight the space of square-zero matrices may seem quite innocent and not very interesting, but the $B$-equivariant geometry of this space contains information about classically studied objects, such as Schubert classes in cohomology of flag varieties.
Note that $\Hom(\C^n,\C^n)$ can be embedded as square-zero matrices of the size $2n\times 2n$
$$A\mapsto \Big(\begin{matrix}0A\\[-0.1cm]0\,0\end{matrix}\Big)\,.$$
This was already noticed by \cite[\S2,2]{KnutsonZinn:2014}.
The Borel orbit of the big matrix consists of matrices of the same shape with $A$ replaced by $B_1AB_2$, where $B_1$ and $B_2$ are upper-triangular $n\times n$ matrices.
The fundamental classes in the equivariant cohomology of $\Hom(\C^n,\C^n)$ are the double Schubert polynomials.
Thus  cohomological properties of $\B$ orbits of square-zero matrices contain information about classical Schubert calculus.
In our approach two sets of variables of the double Schubert polynomials have equal role and can be exchanged by the Hecke action. Exchanging such variables corresponds to leaving the upper-right block. For example if $n=2$, in the doubled dimension we have three rank 2 upper triangular orbits:

$$ \arraycolsep=1.7pt\def\arraystretch{1}
A_1=\hbox{\footnotesize$\left(\begin{array}{cccc}
\cline{3-4}
0 &0 &\multicolumn{1}{|c}{\bf 1}  & \multicolumn{1}{c|}{\bf 0} \\
0& 0&\multicolumn{1}{|c}{\bf 0}  & \multicolumn{1}{c|}{\bf 1} \\
\cline{3-4}
0 & 0 & 0 & 0 \\
0 & 0 & 0 & 0 \\
\end{array}
\right)$}\,,\qquad A_2=\hbox{\footnotesize$\left(\begin{array}{cccc}
\cline{3-4}
0 &0 &\multicolumn{1}{|c}{\bf 0}  & \multicolumn{1}{c|}{\bf 1} \\
0& 0&\multicolumn{1}{|c}{\bf 1}  & \multicolumn{1}{c|}{\bf 0} \\
\cline{3-4}
0 & 0 & 0 & 0 \\
0 & 0 & 0 & 0 \\
\end{array}
\right)$}\,,\qquad
A_3=\hbox{\footnotesize$\left(\begin{array}{cccc}
\cline{3-4}
0 &1 &\multicolumn{1}{|c}{\bf 0}  & \multicolumn{1}{c|}{\bf 0} \\
0& 0&\multicolumn{1}{|c}{\bf 0}  & \multicolumn{1}{c|}{\bf 0} \\
\cline{3-4}
0 & 0 & 0 & 1 \\
0 & 0 & 0 & 0 \\
\end{array}
\right)$}\,. $$

\noindent
The first matrix belongs to the minimal rank 2 orbit and it corresponds to the 0-dimensional Schubert cell in the flag variety $\Fell(2)\simeq\PP^1$. The second one corresponds to the open Schubert cell in $\Fell(2)$. The third one belongs to an {\it exterior} orbit, not visible in cohomology of $\PP^1$. Beside these three orbits we have orbits of rank 1 and 0. There is an action of permutations on the set of the $\B_{2n}$-orbits. In the presented example the simple reflections act as follows: $$s_1\cdot A_1=s_3\cdot A_1=A_2\,,\qquad s_2\cdot A_1=A_3\,.$$
This action\footnote{In our setup this is a partial action. The orbits of $A_2$ and $A_3$ are maximal among the upper-triangular orbits. The matrices $s_2\hskip-2pt\cdot\hskip-2pt A_2$,~ $s_1\hskip-2pt\cdot\hskip-2pt A_3$ and $s_3\hskip-2pt\cdot\hskip-2pt A_3$  are not upper triangular. The resolution of \cite{BenderPerrin} does not assume upper-triangularity of the $B$-orbit. Nevertheless we do not want to extend the exposition in the current paper.}  agrees with the left and right actions of permutations on the set of Schubert cells.
From our point of view right and left operations belong to the same family. Right operations are obtained by conjugation with the block matrix $\Big(\begin{matrix}I\, 0\\[-0.1cm]0\,\sigma\end{matrix}\Big)$ and left operations are those obtained by conjugation with $\Big(\begin{matrix}\sigma\,0\\[-0.1cm]0\,I\end{matrix}\Big)$. Additionally w have the middle operation which destroys the block-upper-triangular structure.
For an extended discussion about right and left Demazure-Lusztig operations see the introduction of \cite{MihalceaNaruse}.
We hope that our approach will shed a new light in on theory of characteristic classes in the classical situation.
\medskip

Quite analogously to Schubert polynomials one can find interpretation of the trigonometric weight function of Rim\'anyi-Tarasov-Varchenko. We describe here only the case of the full flag variety. This time we present the flag variety as the quotient $\Fell(n)=\Hom(\C^{n-1},\C^n)/\hskip-3pt/\B_{n-1}$ and embed $\Hom(\C^{n-1},\C^n)$ into the space of upper-triangular matrices of the size $(2n-1)\times(2n-1)$. Surprisingly the weight functions are exactly equal to the motivic Chern class of orbits (up to the factor corresponding to $\B_{n-1}$).
\smallskip

At the end we would like to remark that the classes of $B$-orbit closures of square-zero matrices were studied by
Di Francesco and  Zinn-Justin in \cite{FranZinn:2005,FranZinn:2005ne}, later by Knutson and Zinn-Justin in \cite{KnutsonZinn:2007,KnutsonZinn:2014}. The results of their work include our formulas for fundamental classes in (usual) equivariant cohomology. The main feature of the mentioned papers is certain Temperley-Lieb algebra, an extension of Hecke algebra, which governs the fundamental classes of square-zero $B$-orbits in the Lie algebra, as well as certain classes in the resolutions of the nilpotent orbit. The relevant formalism of Temperley-Lieb algebra action for K-theory was introduced in \cite{Zinn:2018}. The structure considered there is quite rich. The classes obtained by the Temperley-Lieb action are related to our fundamental classes in K-theory, although the exact relation is involving, see [\emph{loc.~cit.} \S5.4].
Novelty of our results is computation of Chern-Schwartz-MacPherson an motivic Chern classes for square-zero Borel orbits. Moreover geometric interpretation of Tarasov-Rim\'anyi-Varchenko weight function is new. We concentrate on the   trigonometric weight function and its relation with motivic Chern classes of $B_{n-1}\times B_n$-orbits in $\Hom(\C^{n-1},\C^n)$. The parallel statement for homological version follows by well known specialization.
\medskip

Content: After introducing the notation in \S\ref{section:Parametrization}  we provide in \S\ref{section:low_rank_orbits} geometric
description of  orbits corresponding to the  involutions
consisting of one or two transpositions. The general case is treated in the main part of the paper. In  \S\ref{sec:ConBePe} we describe the resolution of singularities of
closures of  orbits by specializing the construction of Bender and Perrin.
The formulas allowing to compute characteristic classes by the action of a suitable Hecke algebra are given in \S\ref{sec:indukcja1} and \S\ref{sec:indukcja2}.
In the remaining sections we show connection with classical Schubert calculus and the theory of weight functions.
\medskip

We would like to thank the Referee for suggesting
many improvements of the manuscript.

\tableofcontents

\section{Parametrization of Borel orbits in $\NN$}
\label{section:Parametrization}
 Let $\NN$ be the space of (strictly) upper triangular $n\times n$
matrices with  complex coefficients. Hence $\NN$ is an affine space of dimension $\frac{(n-1)\cdot n}{2}$. The square-zero matrices form an affine subvariety  \(\{A \in \NN \setsep
A^2=0\} \subseteq \NN\). Let $\GL_n=\GL_n(\C)$ be the general linear
group  and \(\B=\B_n\) be the standard Borel subgroup
of \(\GL_n\), consisting of upper triangular matrices.
The Borel
group acts on $\NN$ by conjugation. It is known that for \(n \geq 6\) the
number of \(\B_n\) orbits in $\NN$ is infinite, while the number of square-zero $\B$-orbits  is finite.
We review now the combinatorial description provided
by Melnikov in \cite[\S1]{Melnikov}.
\medskip

The  set of \(\B_n\)-orbits is in bijection with
involutions in the symmetric group $\Sn$.
Let us fix \(w \in \Sn\).
If \(w\) is an involution then it can be
written uniquely as a composition of disjoint transpositions \(w=\trans{i_1}{j_1} \trans{i_2}{j_2} \dots
\trans{i_k}{j_k}\) for \(i_1<i_2<\dots <i_k\) and \(i_s<j_s\) for any
\(s=1,2,\dots,k\).
  We always assume that involutions are written in this order
  without making explicit appealing to this.	
Let us denote the set of all involutions in
\(\Sn\) by \(\Inv_n\). For any \(w \in \Sn\) let
 \(M_w \in \GL_n\) be the permutation matrix. For $w\in\Inv_n$ we define the matrix \(N_w\) by erasing the lower-triangular and diagonal part from
\(M_w\), that is
\[
  (N_w)_{ij}=
  \begin{cases}
    0, & \text{if } i \geq j, \\
    (M_w)_{ij}, & \text{otherwise.}
   \end{cases}
 \] For example, take \(w=\trans{1}{3} \trans{2}{4} \in \Inv_{4}\). Then
{$$M_w=\begin{pmatrix}0\ 0\ 1\ 0\ 0\\[-0.1cm]0\ 0\ 0\ 1\ 0 \\[-0.1cm]1\ 0\ 0\ 0\ 0\\[-0.1cm]0\ 1\ 0\ 0\ 0\\[-0.1cm]0\ 0\ 0\ 0\ 1\end{pmatrix}\,,\qquad N_w=\begin{pmatrix}0\ 0\ 1\ 0\ 0\\[-0.1cm]0\ 0\ 0\ 1\ 0\\[-0.1cm]0\ 0\ 0\ 0\ 0\\[-0.1cm]0\ 0\ 0\ 0\ 0\\[-0.1cm]0\ 0\ 0\ 0\ 0\end{pmatrix}.$$}

 \begin{theorem}[{\cite[2.2]{Melnikov}}]For each square-zero $\B_n$-orbit
   in $\O\subset\NN$ there exists a unique
   involution \(w \in \Inv_n\) such that \(\O=\O_{N_w}\).
 \end{theorem}

In the sequel we will just write \(\O_w\) as a shorthand for  $\O_{N_w}$.
We will often refer to

\begin{lemma}[{\cite[\S3]{Melnikov}}]
  \label{lem:orbit_dim}
  For \(w = \trans{i_1}{j_1} \trans{i_2}{j_2} \dots \trans{i_m}{j_m} \in
  \Inv_n\) one has
\begin{equation} \label{eq:mel_borel_dim}
    \dim \left(\O_w\right) = mn + \sum_{k=1}^m (i_k-j_k)-\sum_{k=2}^mr_k(w)
  \end{equation}
where for \(k \geq 2\) we put
\begin{equation}\label{def_r}r_k(w)=\sharp \{\ell<k \setsep
 j_{\ell} < j_k\} + \sharp \{\ell < k \setsep j_{\ell} < i_k\}\,.\end{equation}
\end{lemma}
Quick use of the above formula is a bit problematic. There is a graphical
representation of the involution, from which one can read easily the dimension of the orbit. We draw arcs connecting the transposed numbers.   For example if $w=\trans 2 6 \trans 4 7\in \mathcal S_8$ the corresponding picture is the following:
\\

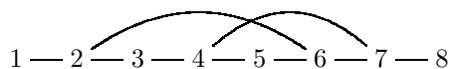
$$
\xymatrix@-1pc{
1\ar@{-}[r]&
2\ar@{-}[r]\ar@{-}@/^1.5pc/[rrrr]&
3\ar@{-}[r]&
4\ar@{-}[r]\ar@{-}@/^1.5pc/[rrr]&
5\ar@{-}[r]&
6\ar@{-}[r]&
7\ar@{-}[r]&
8}
$$
The dimension of the orbit is equal to
$$m(n-m)-\sharp(\text{intersections})-\sum_\text{fixed points}\sharp(\text{arcs passing over})\,.$$
In our example the dimension of the orbit $\O_{\trans 2 6\trans 4 7}$ is equal to~~
$2(8-2)-1-(0+1+2+0)=8$.
We refer for details in \cite{Melnikov-patern}.
\section{Geometric description of lower rank orbits}
\label{section:low_rank_orbits}

\subsection{Orbits corresponding to transpositions}
Let us fix \(1 \leq k<\ell \leq n\)
and consider transposition \(w=\trans{k}{\ell} \in \Inv_n\).  Any matrix belonging to \(\O_w\)  satisfies three conditions
\begin{enumerate}
\item it is of the block form
$
  \begin{pmatrix}
    0 & A \\
    0 & 0
  \end{pmatrix}
$ with \(A\) of the size \(k \times (n-\ell+1)\)
\item the rank of $N_w$ is one.
\item the entry $(N_w)_{k,\ell}$ is non-zero.
\end{enumerate}
Let \(X_w\) be the  closure of
\(\O_w\) in $\NN$.
\begin{prop}
  \label{prop:rank_conditions_trans}
   The variety \(X_w\) is defined by the conditions 1. and 2. above.
\end{prop}

\begin{proof} The variety $Y$ defined by 1. and 2. (or equivalently ${\rm rk}(A)=1$) is a classical rank locus. It is irreducible  of dimension
$n+k-\ell$.
The orbit $\O_{\trans k \ell }$ is of the same dimension by Lemma \ref{lem:orbit_dim}. Since it is contained in $Y$, its closure is equal to $Y$.\qedhere
\end{proof}
The orbit closure $X_{\trans k \ell}$ admits a resolution of singularities which is the bundle $\O(-1)^{\oplus(n-\ell-1)}$ over $\PP(\C^k)$. Here $\C^k$ is identified with the space of the columns having zeros at the positions below $k$. The resolution sends the tuple of proportional vectors $v_\ell,v_{\ell+1},\dots,v_n$ (lying in the same fiber of $\O(-1)$) to the matrix with columns  equal to $v_i$  for $i\geq \ell$ and the remaining columns are zero.

\subsection{Orbits corresponding to double transpositions} Let
\(w=\trans{i}{j}\trans{k}{\ell} \in \Inv_n\). We have 3 cases of the value of
\(r_2(w)\). Namely,  \(r_2(w)\) is equal to $0$, $1$ or $2$.
Clearly $X_{\trans i j \trans k \ell}$ is contained in the algebraic sum
$$X_{\trans i j \trans k \ell}\subset X_{\trans i j}+X_{\trans k \ell}
=\{x+y~|~x\in X_{\trans i j},~y\in X_{\trans k \ell}\}$$
but in general there is no equality.

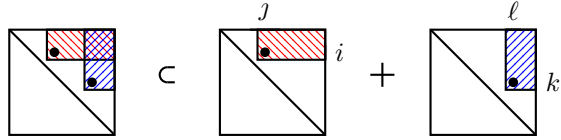
\hfil\begin{tikzpicture}[thick,scale=0.2]
\draw(0,0) rectangle (7,7);
\draw(7,0)--(0,7);
\draw[pattern=north west lines, pattern color=red] (2.5,7) rectangle (7,5);
\draw[pattern=north east lines, pattern color=blue] (5,7) rectangle (7,3);
\draw[fill=black] (5.5,3.5) circle (0.25);
\draw[fill=black] (3,5.5) circle (0.25);
\draw(10.5,4.5) arc (90:270:0.5);
\draw(10.5,4.5)--(11,4.5);
\draw(10.5,3.5)--(11,3.5);
\draw(14,0) rectangle (21,7);
\draw(21,0)--(14,7);
\draw[pattern=north west lines, pattern color=red] (16.5,7) rectangle (21,5);
\draw(16,7) -- (18,7) node [above,text centered,midway]
{$j$};
\draw(21,5) -- (21,6) node [right,text centered,midway]
{$i$};
\draw[fill=black] (17,5.5) circle (0.25);
\draw(24,4)--(25.6,4);
\draw(24.8,3.2)--(24.8,4.8);
\draw(28,0) rectangle (35,7);
\draw(35,0)--(28,7);
\draw[pattern=north east lines, pattern color=blue] (33,7) rectangle (35,3);
\draw(33,7) -- (34,7) node [above,text centered,midway]
{
$\ell$
};
\draw(35,3) -- (35,4) node [right,text centered,midway]
{
$k$
};
\draw[fill=black] (33.5,3.5) circle (0.25);
\end{tikzpicture}
\medskip

\def\kwadrat#1#2#3#4{
\begin{tikzpicture}[thick,scale=0.2]
\draw(0,0) rectangle (7,7);
\draw(7,0)--(0,7);
\draw[pattern=north west lines, pattern color=red] (#2,7) rectangle (7,7-#1);
\draw[pattern=north east lines, pattern color=blue] (#4,7) rectangle (7,7-#3);
\draw[fill=black] (#2+0.5,7.5-#1) circle (0.25);
\draw[fill=black] (#4+0.5,7.5-#3) circle (0.25);
\end{tikzpicture}}

\bigskip

\noindent{\bf 0. The case \(r_2(w)=0\).} For the pattern
$
\xymatrix@-1pc{
_i\ar@{--}[r]
\ar@{-}@/^0.8pc/[rrr]
&
_k\ar@{--}[r]
\ar@{-}@/^0.7pc/[r]&
_\ell\ar@{--}[r]&
_j
}
$ (the dashed lines indicate, that the subsequent vertices do not have to be adjacent) the matrices  belonging to $X_w$ are nested
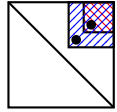
\begin{center}\kwadrat 2534 \end{center}
The bigger rectangle corresponds to the $k\times(n-\ell+1)$ matrix, the smaller one to the $i\times (n-j+1)$-matrix.
\medskip

We describe a partial resolution of the variety $X_w$. Let $\Fell(1,2;\C^k)$ be the 2-step flag variety consisting of lines and planes in $\C^k$.
Let \(\pi \colon
\C^k \rightarrow \C^{k-i} \) be the projection onto last $(k-i)$-coordinates. Let
\(\Omega \subseteq \Fell(1,2;\C^k)\) be the Schubert variety defined by the condition
$$\Omega=\{(L,P)\in \Fell(1,2;\C^k) \setsep \dim\pi(P)\leq 1\}=\{(L,P)\in \Fell(1,2;\C^k) \setsep \dim (P\cap \C^i)\geq 1\}\,.$$
We have the tautological
bundles of lines \(\mathcal{L}\) and planes \(\mathcal{P}\) over the flag variety $\Fell(1,2,n)$. Consider the restricted bundle $$\mathcal{E}=\left(\mathcal{L}^{\ell-j}
\oplus \mathcal{P}^{n-j-1}\right)_{|\Omega}\,.$$

Let $p:\mathcal E\to \NN$ be the map which sends the coordinates of $\mathcal E$ to the columns of the matrix. The variety $X_{\trans i j}+X_{\trans k \ell}$ is contained in the image of $p$. Counting the dimensions (Lemma \ref{lem:orbit_dim}) we see that
$$X_{\trans i j \trans k \ell}=X_{\trans i j}+X_{\trans k \ell}=p(\mathcal E)\,.$$
The map \(p \colon \mathcal{E} \rightarrow
X_w\) is proper but it is not a resolution of singularities of \(X_w\) since \(\mathcal{E}\)
is a bundle over the singular base \(\Omega\).
Resolving the singularities of the Schubert variety $\Omega$ we obtain a resolution of $X_w$.

\begin{cor} If $w=\trans ij\trans k\ell$ and $i<k<\ell<j$ then the orbit closure $X_w$ is defined by rank conditions.\end{cor}

We will not give a detailed explanation of the above statement, since in general the ideal defining $X_w$ is given in \cite[Sec. 3, Theorem \S3.5]{Melnikov-ideal}. In this case the ideal is generated by the conditions which can be expressed in terms of geometric position of the ``north-east'' rectangles\footnote{For example the  ``noth-east'' rectangle $NE(i,j)$ determined by the pair $(i,j)$ consist of the entries $(i',j')$ with $1\leq i'\leq i$ and $j\leq j'\leq n$.} $NE(i,j)$ and $NE(k,\ell)$ determined by the pairs $(i,j)$ and $(k,\ell)$:
\begin{itemize}
\item ${\rm rk}(A)\leq 2$;
\item the $2\times 2$ minors of the submatrices not intersecting with the intersection $NE(i,j)\cap NE(k,\ell)$ vanish;
\item the entries of the matrix which do not lie in any rectangle are zero.
\end{itemize}

\bigskip

\noindent{\bf 1. The case \(r_2(w)=1\).} For the pattern
$
\xymatrix@-1pc{
_i\ar@{--}[r]
\ar@{-}@/^0.7pc/[rr]
&
_k\ar@{--}[r]
\ar@{-}@/^0.7pc/[rr]&
_j\ar@{--}[r]&
_\ell
}
$
the matrices belonging to $X_w$ are of the shape:
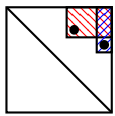
\begin{center}\kwadrat 2436 \end{center}
Note that both rectangles are contained in a bigger one, which is entirely located above the diagonal.
Let us consider the (smooth)
 Schubert variety \[ \Omega=\{(L,P) \in \Fell(1,2;\C^k) \setsep L \subseteq \C^i\}\,.\]
The variety \(\Omega\) fibers  over $\PP(\C^i)$ with the fiber isomorphic to $\PP(\C^{k-1})$, therefore $\dim\Omega=
  i+k-3$. Similarly as in the previous case we set
$$\mathcal{E}=\left(\mathcal{L}^{\ell-j} \oplus\mathcal{P}^{n-\ell+1}\right)_{|\Omega}\,.$$
Here \[\dim(\mathcal{E})=i+k-3+(\ell-j)+2(n-\ell+1)=2n+(i-j)+(k-\ell)-1\,.\]
As before we define the projection map $p:\mathcal E\to\NN$. By the construction $X_w\subset p(\mathcal E)$. The variety $p(\mathcal E)$ is irreducible and  by Lemma \ref{lem:orbit_dim}~~ $\dim(X_w)=2n+(i-j)+(k-\ell)-1$ as well. Thus $X_w= p(\mathcal E)$ and $p$ is a resolution of singularities.

\begin{cor} If $w=\trans ij\trans k\ell$ and $i<k<j<\ell$ then the orbit closure $X_w$ is defined by rank conditions.\end{cor}
\bigskip
\noindent{\bf 2. The case \(r_2(w)=2\).} For the pattern
$
\xymatrix@-1pc{
_i\ar@{--}[r]
\ar@{-}@/^0.7pc/[r]
&
_j\ar@{--}[r]
&
_k\ar@{--}[r]\ar@{-}@/^0.7pc/[r]&
_\ell
}
$
matrices belonging to $X_w$ are of the shape:
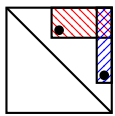
\begin{center}\kwadrat 2356 \end{center}
This shape of matrices in \(X_w\)
is analogous to the previous case. The difference is that the North-East
rectangle spanned by coordinate \((k,j)\) does not lie in the strictly upper
triangular part of \(n\times n\) matrices. Rank conditions do not suffice to
define \(X_w\)  in this case. One equation is missing, $X_w$ is a hypersurface in $p(\mathcal E)$. It turns out we just need to intersect
the rank conditions variety with the condition $A^2=0$.
Geometrically this means that we do not take the whole bundle \(\mathcal{E}\) (defined
analogously as in the previous cases) but we have to restrict to a naturally
defined quadric \(Q \subseteq \mathcal{E}\). The singularities in this case are quadratic and they are easy to resolve.

\begin{cor} If $w=\trans ij\trans k\ell$ and $i<j<k<\ell$ then the orbit closure $X_w$ is defined by rank conditions and the condition $A^2=0$.\end{cor}

Further investigation by hand of  $X_w$ singularities is troublesome. We stop here our consideration. For a general description of the ideal of $X_w$ see \cite{Melnikov-ideal}. The precise form of the ideal of $X_w$ is irrelevant for our purposes. We will be primarily interested in resolution of singularities of $X_w$.
As we can see the resolution of $\B$-orbit closure is intimately related to resolving singularities of Schubert cells.

\section{Inductive construction of resolution}
\label{sec:ConBePe}
In the article \cite{BenderPerrin} Bender and Perrin proved a general result about resolution of
singularities for $\B$-orbits in 2-nilpotent cones in simple simply-laced groups (i.e. in the type ADE).
In this section, we will give a precise form of this resolution
in the special case \(G=\SL_n(\C)\). In subsequent sections we will apply the resolution obtained here to compute fundamental classes of orbit closures and characteristic classes of orbits.

\begin{theorem}[{\cite[Corollary 3.2.1]{BenderPerrin}}] \label{Bender_Perrin_resolution}
  Let \(G\) be a simple simply-laced group and $B$ its Borel subgroup. Let \(x \in \mathfrak{g}\) be
  nilpotent element of height \(2\) and let \(X\) be its \(B\)-orbit closure. Then, there exists a sequence of minimal parabolic
  subgroups \(P_1,\dots,P_m\) and a vector space \(Y\), such that
  \[
    f \colon P_1 \times^B \dots \times^B P_m \times^B Y \rightarrow X,
    \hspace{1cm} [p_1,\dots,p_m,x] \mapsto p_1\dots p_m \cdot x,
  \] is a resolution of singularities.
  \end{theorem}

In our context, \(G=\GL_n(\C)\), but results for nilpotent orbits apply.

\subsection{The minimal orbit of a fixed rank}

All the square-zero matrices of a fixed rank are conjugate. Each adjoint orbit $\GL_n \cdot N_w$ contains a minimal $\B_n$ orbit. It is closed in $\GL_n \cdot N_w$. It is elementary to check that the following orbit is minimal, for example computing its dimension by  Lemma \ref{lem:orbit_dim}.

\begin{prop} The minimal orbit in $\GL_n \cdot N_w$ is of the form $N_{w_m}$, where $m=rk(N_w)$ and
$$w_m=(1,n-m+1)(2,{n-m+2})\dots(m,n)\,.$$
\vskip10pt
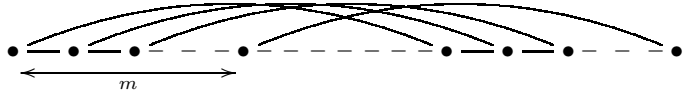
$$
\xymatrix@-1pc@R-1pc{
\bullet\ar@{-}[r]\ar@{-}@/^1.5pc/[rrrrrrrr]&
\bullet\ar@{-}[r]\ar@{-}@/^1.5pc/[rrrrrrrr]&
\bullet\ar@{--}[rr]\ar@{-}@/^1.5pc/[rrrrrrrr]&&
\bullet\ar@{--}[rrrr]\ar@{-}@/^1.5pc/[rrrrrrrr]&
&&&
\bullet\ar@{-}[r]&
\bullet\ar@{-}[r]&
\bullet\ar@{--}[rr]&&
\bullet\\
\ar@{<->}[rrrr]_m
&&&&}
$$

The closure of $\O_{w_m}$ is the linear space $X_{w_m}$ consisting of the matrices $A$ having the entries $a_{ij}=0$ for $j+i\leq 2n-m$.
\vskip 10pt
\hfil 
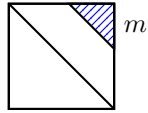
\begin{tikzpicture}[thick,scale=0.2]
\draw(0,0) rectangle (7,7);
\draw[pattern=north east lines, pattern color=blue] (7,7)--(4,7)--(7,4)--cycle;
\draw(7,0)--(0,7);
\draw(0,0) rectangle (7,7);
\draw(7,4) -- (7,7) node [right,text centered,midway]
{$m$};
\end{tikzpicture}
\end{prop}

\subsection{A combinatorial lemma}

We will need the following lemma:
\begin{lemma} \label{sequence_lemma}
  Let \(w=\trans{i_1}{j_1} \dots \trans{i_m}{j_m} \in \Sn\) be an involution. Let
  \(k_1<k_2<\dots<k_{n-2m}\) be the fixed points of \(w\). Suppose \(\pi_w \in
  \Sn\) is the permutation, written in one-line notation as
  \[
    \pi_w=i_1,i_2,\dots,i_m,k_1,k_2,\dots,k_{n-2m},j_1,j_2,\dots,j_m
  \] i.e. \(\pi_w(1)=i_1,\pi_w(2)=i_2,\dots,\pi_w(n)=j_m\). Then
\begin{equation}\label{sprzeganie}\pi_w \cdot N_{w_m}=M_{\pi_w} N_{w_m}
  M_{\pi_w}^{-1}=N_w\end{equation}
Moreover
  \begin{equation}\label{dim_length}
    \dim(\O_w)=\ell(\pi_w)+\tfrac{m(m+1)}{2}
  \end{equation}
\end{lemma}

\begin{proof}
The first assertion is clear. Let us compute the length of $\pi_w$ counting the number of the inversions in $\pi_w$. The set $\{1,2,\dots,n\}$ is the sum of three intervals: $I=[1,m]$, $K=[m+1,n-2m]$ and $J=[n-2m+1,n]$.
The permutation $\pi_w$ preserves the orders of the first two sets.
Therefore
\begin{align*}\ell(\pi_w)&=in(I,K)+in(I,J)+in(K,J)+in(J,J)\\
&=in(I,K\cup J)+in(I\cup K, J)-in(I,J)+in(J,J)\,,\end{align*}
where
$$in(X,Y)=\sharp\{(x,y)\in X\times Y \setsep x<y,~\pi_w(x)>\pi_w(y)\}\,.$$
Also note that if $\ell\geq k$ then $j_\ell>i_\ell\geq i_k$.
Therefore
$$in(I,J)=\sharp \{\ell < k \setsep j_{\ell} < i_k\}\,.$$
$$in(J,J)=\sharp \{\ell < k \setsep j_{\ell} > j_k\}=\tfrac{m(m-1)}2
-\sharp \{\ell < k \setsep j_{\ell} < j_k\}\,.$$
By the definition \eqref{def_r}
$$\sum_{k=2}^m r_k(w)=\tfrac{m(m-1)}2+in(I,J)-in(J,J)\,.$$
Moreover
$$in(I,K\cup J)=\sum_{k=1}^m i_k-k=\sum_{k=1}^m i_k-\tfrac{m(m+1)}2\,.$$
Let $(j_1'>j_2'>\dots>j_m')$ be the sequence of indices $j_k$ written in the reverse order. We have
$$in(I\cup K,J)=\sum_{k=1}^m (n-j'_k+1)-k=nm-\sum_{k=1}^m j_k-\tfrac{m(m-1)}2\,.$$
Finally we obtain
\begin{align*}\ell(\pi_w)&=\left(\sum_{k=1}^m i_k-\tfrac{m(m+1)}2\right)+
\left(nm-\sum_{k=1}^m j_k-\tfrac{m(m-1)}2\right)-
\left(\sum_{k=2}^m r_k(w)-\tfrac{m(m-1)}2\right)\\&=nm+\sum_{k=1}^m(i_k-j_k)-\sum_{k=2}^m r_k(w)-\tfrac{m(m+1)}2\,.\end{align*}
Conclusion follows from Lemma \ref{lem:orbit_dim}.
\end{proof}
\begin{rmk} The permutation $\pi_w$ defined in the Lemma \ref{sequence_lemma} in general is not the unique one satisfying \eqref{sprzeganie} and \eqref{dim_length}.
Let
$w=\trans 15 \trans 26 \trans 34$
$$
\xymatrix@-1pc{
1\ar@{-}[r]\ar@{-}@/^0.8pc/[rrrr]&
2\ar@{-}[r]\ar@{-}@/^0.8pc/[rrrr]&
3\ar@{-}[r]\ar@{-}@/^0.4pc/[r]&
4\ar@{-}[r]&
5\ar@{-}[r]&
6}
$$
there are three permutations conjugating $N_{w_3}$ to $N_w$: 123564, 132546 and 312456. They all have the minimal length equal to 2.
\end{rmk}

\subsection{Multi-twisted product}

\begin{defn}Let $0<i<n$. \begin{itemize}\item The transposition $(i,i+1)\in \Sn$ is denoted by $s_i$.
\item The subgroup of $\GL_n$ generated by $\B_n$ and $N_{s_i}$ is denoted by $P_i$. It is a minimal parabolic group (properly) containing $\B_n$.\end{itemize}\end{defn}

Now let us define the main protagonist of this section. Let $\piu=(s_{i_1},s_{i_2},\dots ,s_{i_\ell})$ be a reduced word representing $\pi_w$.
We define the multiple twisted product
\[\X_{\piu}=P_{i_1} \times^{B} \dots
  \times^{B} P_{i_\ell} \times^{B} X_{w_m}\,.\]
with respect to the Borel group \(B=\B_n\). In addition $B$ acts on $\X_{\piu}$ and the morphisms
\[
  \begin{matrix}
  \begin{tikzcd}
    \X_{\piu} \arrow[r,"f_{\piu}"]
    \arrow[d,"g_{\piu}"] & \mathfrak{g}=
    M_{n \times n}(\C) \\
    P_{i_1}/B \cong \mathbb{P}^1
  \end{tikzcd}
  &&
  \begin{tikzcd}
    (p_{i_1},\dots,p_{i_\ell},y) \arrow[r, maps to] \arrow[d, maps to] & p_{i_1}\dots
    p_{i_\ell} \cdot y \\
    p_{i_1}B
  \end{tikzcd}
  \end{matrix}
\]
are $B$-equivariant.
The map \(g_{\piu}\) is a fibration with fiber
\(\X_{\piu'}\), where
\(\piu'=(s_{i_2},\dots,s_{i_\ell})\).
 In fact, we have
\(\X_{\piu}=P_{i_1} \times^{B}
\X_{\piu'}\).

\begin{theorem}\label{th:resolution}
  Let \(w=\trans{i_1}{j_1} \dots \trans{i_m}{j_m} \in \Inv_n\) and let \(\pi_w\)
  be as in Lemma \ref{sequence_lemma}.  Let
  \(\piu=(s_{i_1},s_{i_2},\dots,s_{i_\ell})\) be a reduced word representing $\pi_w$. Then
\[f_{\piu}:\X_{\piu}
\to X_w\subset \NN\]
is a $B$-equivariant resolution  of singularities of the closure of $\O_w$.
\end{theorem}

\begin{proof} The map $f_{\piu}$ is a composition of two proper morphism. The first is the Springer resolution $\GL_n\times^{B} \NN\to \mathfrak g$ with the Bott-Samelson resolution of the (bundle over) Schubert variety
$(B\pi B)\times^B X_{w_m}$. Therefore the map  $f_{\piu}$ is proper.
\medskip

The map $f_\piu$ sends the sequence $\underline w=[{s_{i_1}},{s_{i_2}},\dots,{s_{i_m}},N_{w_m}]$ to
$(s_{i_1}s_{i_2}\dots s_{i_m})\cdot N_{w_m}=N_w$. (Here we identified the transposition $s_{i_a}$ with its matrix $M_{s_{i_a}}$ to avoid staircase indices.) Therefore the whole orbit $\O_w$ is contained in the image. By Lemma \ref{sequence_lemma} $\dim \O_w=\dim \X_{\underline \pi}$, therefore the  map $f_{\piu}$ restricted to the orbit of $\underline w$ is a covering. On the other hand (by an elementary argument) the stabilizer of $N_w$ in $B$ is connected. Therefore this covering is trivial. Hence $f_{\underline \pi}$ is a resolution of singularities.\end{proof}

We have described in a constructive way the resolution of $X_w$ of
Theorem \cite[Corollary 2.4.3]{BenderPerrin}. Compare [Lemma 7.3.1, \emph{loc.~cit.}].
\medskip

\begin{rmk} Suppose $\piu=s_{i_1}s_{i_2}\dots s_{i_\ell}$ is a reduced word. Then for $0\leq k\leq\ell$ the composition $\tau
=s_{i_k}\circ s_{i_{k+1}}\circ \dots \circ s_{i_\ell}$ is a permutation which preserves upper-triangularity of $N_{w_m}$, i.e. $\tau N_{w_m}\tau^{-1}$ is upper triangular. This follows from the fact that if $a<b$ and $\pi(a)<\pi(b)$ then the inequality is preserved for the truncated word.
Suppose $\X_\piu\to X_w$ is a resolution of singularities or equivalently $N_w=\piu\cdot N_{w_m}$ and  $\dim(\X_\piu)=\dim( X_w)$.
Let $\underline\tau=(s_{i_k},s_{i_{k+1}},\dots ,s_{i_\ell})$.
Then at each step the matrix $N=\tau\cdot N_{w_m}$ is upper-triangular and $\X_{\underline\tau}$ is a resolution of the $\B$-orbit of $N$.
 \end{rmk}
\medskip

\begin{ex} Let \label{ex1634}
$w=\trans 16\trans 34\in \Inv_7$.
The permutation given by Lemma \ref{sequence_lemma} is equal to
$\piu=s_2s_6s_4s_5s_6$.
\def\kwadrat#1#2#3#4{
\begin{tikzpicture}[thick,scale=0.2]
\draw(0,0) rectangle (8,8);
\foreach \x in {0,1,2,3,4,5,6,7}
  \draw[fill=gray] (\x ,8-\x ) rectangle (\x+1 ,7-\x );
\draw[pattern=north west lines, pattern color=red] (#2,8) rectangle (8,8-#1);
\draw[fill=black] (#2+0.5,8.5-#1) circle (0.25);
\draw[pattern=north east lines, pattern color=blue] (#4,8) rectangle (8,8-#3);
\draw[fill=black] (#4+0.5,8.5-#3) circle (0.25);
\end{tikzpicture}}

Let us analyze all the intermediate steps obtained inductively which lead to the resolution of the orbit closures. We call whole process \emph{evolution of the dots}.

\vskip10pt

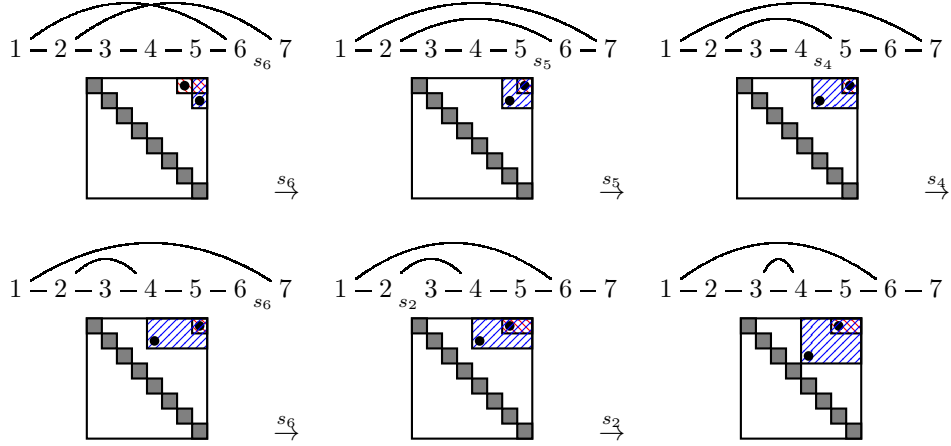
$$\xymatrixcolsep{1pc}\begin{matrix}
\xymatrix@-0.5pc{
1\ar@{-}[r]\ar@{-}@/^1.5pc/[rrrrr]&
2\ar@{-}[r]\ar@{-}@/^1.5pc/[rrrrr]&
3\ar@{-}[r]&
4\ar@{-}[r]&
5\ar@{-}[r]&
6\ar@{~}[r]_{s_6}&
7}
&

\xymatrix@-0.5pc{
1\ar@{-}[r]\ar@{-}@/^1.5pc/[rrrrrr]&
2\ar@{-}[r]\ar@{-}@/^1pc/[rrrr]&
3\ar@{-}[r]&
4\ar@{-}[r]&
5\ar@{~}[r]_{s_5}&
6\ar@{-}[r]&
7}
&
\xymatrix@-0.5pc{
1\ar@{-}[r]\ar@{-}@/^1.5pc/[rrrrrr]&
2\ar@{-}[r]\ar@{-}@/^1pc/[rrr]&
3\ar@{-}[r]&
4\ar@{~}[r]_{s_4}&
5\ar@{-}[r]&
6\ar@{-}[r]&
7}
\\
\phantom\to\hfill\kwadrat 1627\hfill\stackrel{s_6}\to&
\phantom\to\hfill\kwadrat 1726\hfill\stackrel{s_5}\to&
\phantom\to\hfill\kwadrat 1725\hfill\stackrel{s_4}\to&
\\ \\ \\
\xymatrix@-0.5pc{
1\ar@{-}[r]\ar@{-}@/^1.5pc/[rrrrrr]&
2\ar@{-}[r]\ar@{-}@/^1pc/[rr]&
3\ar@{-}[r]&
4\ar@{-}[r]&
5\ar@{-}[r]&
6\ar@{~}[r]_{s_6}&
7}
&
\xymatrix@-0.5pc{
1\ar@{-}[r]\ar@{-}@/^1.5pc/[rrrrr]&
2\ar@{~}[r]_{s_2}\ar@{-}@/^1pc/[rr]&
3\ar@{-}[r]&
4\ar@{-}[r]&
5\ar@{-}[r]&
6\ar@{-}[r]&
7}
&
\xymatrix@-0.5pc{
1\ar@{-}[r]\ar@{-}@/^1.5pc/[rrrrr]&
2\ar@{-}[r]&
3\ar@{-}[r]\ar@{-}@/^1pc/[r]&
4\ar@{-}[r]&
5\ar@{-}[r]&
6\ar@{-}[r]&
7}
\\
\phantom\to\hfill\kwadrat 1724\hfill\stackrel{s_6}\to&
\phantom\to\hfill\kwadrat 1624\hfill\stackrel{s_2}\to&
\kwadrat 1634
\end{matrix}
$$
\end{ex}

\subsection{Torus fixed points in the  resolution}\label{sec:fix}
Let $\T=\T_n\times \C^*$, where $\T_n\subset \B_n$ is the maximal torus consisting of the diagonal matrices.
The torus $\T_n$ acts on $\NN$, $\X_{\piu}$, $X_w$, $P_i/B$ etc. as a subgroup of the Borel group. The factor $\C^*$ of $\T$ acts
on $\NN$ as the scalar multiplication.
The resolution $\X_{\piu}$ is a fiber bundle over the Bott-Samelson variety
$$Z_{\piu}=P_{i_1} \times^{B}P_{i_2} \times^{B} \dots
  \times^{B} P_{i_\ell}/B\,,$$
which is a resolution of the Schubert variety $B\pi B/B\subset \GL_n/B$, see \cite[\S2.2]{BrionKumar}. The action of the factor $\C^*$ lifts to
$\X_{\piu}$ as the scalar multiplication in the fibers of the projection $\X_\piu\to Z_\piu$.
\medskip

The fixed points of the torus $\T_n$ (as well as of $\T$) on $\X_{\piu}$ are contained in the zero section, i.e. in
$Z_{\piu}$. They are indexed by the subwords of $\piu$, see \cite[\S3.2]{RW}. Hence $|\X_{\piu}^\T|=2^m$.
Our further computation is an induction based on the fibration $g_{\piu}:\X_{\piu}\to P_i/B\simeq\PP^1$.
The fixed point $B/B\in P_i/B$ will be denoted by $p_0$ and the other fixed point $s_iB/B$ will be denoted $p_\infty$.
The fiber over $p_0$ is equivariantly isomorphic to $\X_{\underline \pi'}$ where the word $\piu'$ arises by omitting the first letter in  $\piu$. The fiber over $p_\infty$ is isomorphic to $\X_{\underline \pi'}$ too but the action of the torus is twisted.
For $t\in \T$ and $[s_{i_1},x]\in g^{-1}_{\piu}(p_\infty)$
$$t\cdot [s_{i_1},x]=[ts_{i_1},x]=[s_{i_1}s_{i_1} t s_{i_1},x]=[s_{i_1}, s_{i_1} t s_{i_1}x]\,,$$
i.e. the action is twisted by  $s_{i_1}$.

\section{Fundamental classes: localization and induction}
\label{sec:indukcja1}
\subsection{Equivariant push-forward}
We briefly describe the main tool which allows efficient computation of fundamental classes in the situation when a torus is acting and there are finitely many fixed points. The method is based on the equivariant push-forward and may be applied in many contexts, for various cohomology theories. We are primarily interested in ordinary cohomology and K-theory, but for sake of compactness we present the general formula.
\medskip

Suppose $h_\T(-)$ is a $\T$-equivariant version of a complex oriented  cohomology theory in the sense \cite[\S1]{LevineMorel}.
Let $f:M\to N$ be a proper map of smooth complex varieties. Let us assume that the fixed point sets $M^\T$ and $N^\T$ are finite.
From the formal properties of complex oriented theories (see \cite{tomDieck}) it follows that for any class
$\alpha\in h_\T(M)$ and any fixed point $p\in N^\T$
\begin{equation}\label{loc_formula}\frac{f_*(\alpha)_{|p}}{e^h(T_pN)}=\sum_{q\in f^{-1}(p)\cap M^\T}\frac{\alpha_{|q}}{e^h(T_qM)}\,.\end{equation}
Here for a $\T$ representation $V$ the class $e^h(V)\in h_\T(pt)$ denotes the Euler class in the cohomology theory $h(-)$.
For the classical cohomology the Euler class of a one dimensional representation is equal to the weight of the action, i.e.
$$e(L)=\omega\in \Hom(\T,\C^*)\simeq H^2_\T(pt)\,.$$
For K-theory
$$e^K(L)=1-L^*\in {\rm R}(\T)\simeq K_\T(pt)\,.$$
Formula \eqref{loc_formula} holds in the fraction field of $h_\T(pt)$, denoted by $(h_\T(pt))$, therefore
it allows to compute the direct image provided that $h_\T(N)\to (h_\T(pt))\otimes_{h_\T(pt)}h_\T(N)$ is injective. Otherwise the direct image is computed up to the $h_\T(pt)$-torsion. We will apply the localization theorem for $N=\NN$, for which $h_\T(\NN)\simeq h_\T(pt)$.
\medskip

Let $X\subset N$ be a singular variety and let $f:M\to N$ be a proper map, which is a resolution of singularities of $X=f(M)$. The push forward of the unit element $\jeden_M\in h_\T(M)$  plays the role of the fundamental class of $X$ in the theory $h(-)$. In general it depends on the resolution except very few cases. For classical cohomology theory it agrees with the usual notion of the fundamental class. For K-theory it does not depend on the resolution in general, and it has good properties when $X$ has rational singularities, see \cite{Feher}. To compute $f_*(\jeden_M)$ we specialize the formula \eqref{loc_formula} and we obtain the expression
\begin{equation}\label{loc_formula_fun}\frac{f_*(\jeden_M)_{|p}}{e^h(T_pN)}=\sum_{q\in f^{-1}(p)\cap M^\T}\frac{1}{e^h(T_qM)}\,.\end{equation}
In our case, for $X=X_w$, $M=\X_\piu$ and $f=f_{\piu}:\X_{\piu}\to \NN$.

\subsection{The induction} \label{fund-indukcja}
We will use the notation introduced in \S\ref{sec:fix}.
All the formulas below hold in the fraction field of $h_\T(pt)$. The simple reflections $s_i$ act on $h_\T(pt)$ through the automorphisms of the torus. Below $t_j$ denotes the one dimensional representation given by character, which is the projection of the torus on $j$-th coordinate.
\begin{theorem}\label{th:fund-indukcja} For $x\in(h_\T(pt))$ let $$\bett^h_i(x)=\frac x{e^h(t_{i+1}t_i^{-1})}+\frac{s_i(x)}{e^h(t_it_{i+1}^{-1})}\,.$$
Suppose $\piu=(s_{i_1},s_{i_2},\dots,s_{i_\ell})$, $\piu'=(s_{i_2},s_{i_3},\dots,s_{i_\ell})$.
Then
$$\frac{f_{\piu*}(\jeden_{\X_\piu})}{e^h(\NN)}=\bett_{i_1}\left(\frac{f_{\piu'*}(\jeden_{\X_{\piu'}})}{e^h(\NN)}\right)\,.$$
\end{theorem}

\begin{proof}Let $p=p_0$ or $p_\infty$. The tangent space at $q\in g_\piu^{-1}(p)$ fits to the exact sequence
$$T_q\X_{\piu'}\hookrightarrow T_q\X_{\piu}\twoheadrightarrow T_{p}\PP^1\,.$$
The tangent space at $p_0\in P_i/B$ is equal to $t_{i+1}t_i^{-1}$, at the point $p_\infty$ it is dual.
The fixed point set decomposes into two parts. \begin{itemize}
\item The points belonging to $g^{-1}_{\piu}(p_0)$:
here
$$e^h(T_q\X_{\piu})=e^h(T_q\X_{\piu'})e^h(t_{i+1}t_i^{-1})\,.$$
\item The points belonging to $g^{-1}_{\piu}(p_\infty)$: the action of $\T$ at the fiber $g_\piu^{-1}(p_\infty)$ is twisted by  $s_{i_1}$:
$$e^h(T_q\X_{\piu})=s_{i_1}\left(e^h(T_q\X_{\piu'})\right)e^h(t_it_{i+1}^{-1})\,.$$
 \end{itemize}
From the localization  formula \eqref{loc_formula_fun} we obtain
\begin{align*}\frac{f_{\piu*}(\jeden_{\X_\piu})}{e^h(\NN)}&=
\sum_{q\in g_\piu^{-1}(p_0)}\frac 1{e^h(T_q\X_{\piu'})}\cdot\frac 1{e^h(t_{i+1}t_i^{-1})}+
\sum_{q\in g_\piu^{-1}(p_\infty)}\frac 1{s_{i_1}(e^h(T_q\X_{\piu'}))}\cdot\frac 1{e^h(t_{i+1}t_i^{-1})}\\
&=
\frac{f_{\piu'*}(\jeden_{\X_{\piu'}})}{e^h(\NN)}\cdot\frac 1{e^h(t_{i+1}t_i^{-1})}
+s_{i_1}\left(\frac{f_{\piu'*}(\jeden_{\X_{\piu'}})}{e^h(\NN)}\right)\cdot\frac 1{e^h(t_it_{i+1}^{-1})}\,.\end{align*}
\end{proof}

For the classical theory the operator $\bett_i$ specializes to  the divided difference (up to the sign)
\begin{equation}\label{nasze_partial}\bett^H_i(x)=\frac x{t_{i+1}-t_i}+\frac {s_i(x)}{t_i-t_{i+1}}=-\frac{x-s_i(x)}{t_i-t_{i+1}}=-\partial_i(x)\,.\end{equation}
In the case of ordinary cohomology our computation coincides with \cite[\S4.1, Lemma 1]{KnutsonZinn:2007}.
\medskip

For K-theory we obtain a variant of the isobaric divided difference
$$\bett^K_i(x)=\frac x{1-t_it_{i+1}^{-1}}+\frac {s_i(x)}{1-t_{i+1}t_i^{-1}}= \frac{t_{i+1}x-t_is_i(x)}{t_{i+1}-t_i}\,.$$
If we considered the lower-triangular matrices with the action of $B_-$ we would obtain the classical formulas.

It remains to give a formula for the starting point of the induction. The variety $X_{w_m}$ is smooth and closed (a vector subspace), thus its fundamental class in any theory is well defined. It is equal to the product of the Euler classes of coordinates $e^h(u\,t_it_j^{-1})$ for $(i,j)$ not belonging to the corner defining $X_{w_m}$.
\begin{prop}\label{fund-start}The class of the closure of the minimal orbit of rank $m$ divided by the Euler class is equal to
$$\frac{[X_{w_m}]_h}{e^h(\NN)}=\prod_{ n-m+i\leq j\leq n} e^h(u\,t_it_j^{-1})^{-1}\,.$$
In particular
$$\frac{[X_{w_m}]}{e(\NN)}=\prod_{ n-m+i\leq j\leq n} (u+t_i-t_j)^{-1}\,.$$
and
$$\frac{[X_{w_m}]_K}{e^K(\NN)}=\prod_{ n-m+i\leq j\leq n} (1-u\,t_j/t_i)^{-1}\,.$$
\end{prop}

\begin{rmk}\label{non_ambiguity}
We note that  $X_w$ has rational singularities by
\cite[Cor. 0.0.4]{BenderPerrin}. This is clear from the form of the resolution given in Theorem \ref{th:resolution}.
Therefore the following notions coincide:
\begin{itemize}
\item $f_*([\O_{\tilde X_w}])$ for any resolution of singularities $f:\tilde X_w\to \NN$, $f(\tilde X_w)=X_w$,
\item the class of the coherent sheaf $\O_{X_w}$,
\item the specialization of the motivic Chern class $\mC(X_w\subset \NN)_{y=0}$.
\end{itemize}
See \cite[\S14]{WeSEL}, \cite{Feher}.
\end{rmk}
\section{Characteristic classes}
\label{sec:indukcja2}

\def\hb{\text{\textcrh}}
\subsection{The induction}
For a generalized equivariant cohomology theory $h_-(-)$ we define the characteristic classes enlarging the transformation group. For a vector bundle $E
\to X$, possibly equivariant with respect to a group $G$ we define
$$c^h(E)=e^h(E\otimes \C_\hb)\in h_{G\times \C^*}(X)\simeq h_G(X)\otimes h_{\C^*}(pt)\,,$$
where $\C_\hb$ is the natural representation of $\C^*$.
In this section we only consider two generalized cohomology theories: the classical one and K-theory.
Then for a line bundle $L$
\begin{align*}
c^H(L)&=\hb+c_1(L)\\
c^K(L)&=1-(\hb L)^{-1}\,.
\end{align*}
Setting $\hb=1$ in the first case we obtain the usual notion of the Chern class. For K-theory substituting $\hb=-y^{-1}$ we arrive to $\lambda_yE^*$, which is used to construct motivic Chern classes. (See \cite[Rem. 2.4]{FRW} for a similar convention.)
\medskip
Suppose $f:M\to N$ is a proper map of smooth varieties. We define
$$c^h(f:M\to N)=f_*(c^h(TM))\,.$$
For classical homology and K-theory this assignment can be extended to $K^G(Var/\NN)$ --- the Grothendieck ring of $G$-varieties over $\NN$, see \cite{BSY,AMSS}.
\medskip

Let $\X^o_\piu\subset \X_\piu$ be the open $B$-orbit. If $\pi\cdot N_{w_m}=N_w$ and $\dim(\X^o_\piu)=\dim(\O_w)$ then $\X^o_\piu$ is mapped isomorphically to $\O_w$. The orbit $\X^o_\piu$ can be  described geometrically:
$$\X^o_\piu=\left(P_{i_1}\times^B\X^o_{\piu'}\right)\setminus g_\piu^{-1}(p_0)\,.$$
The effect on characteristic classes is the following:
\begin{theorem} \label{char-indukcja} For $x\in(h_\T(pt))$ let
$$A^h_i(x)=\left(\frac {c^h(t_{i+1}t_i^{-1})}{e^h(t_{i+1}t_i^{-1})}-1\right)x+\frac{c^h(t_it_{i+1}^{-1})}{e^h(t_it_{i+1}^{-1})}s_i(x)\,.$$
Suppose $\piu=(s_{i_1},s_{i_2},\dots,s_{i_\ell})$, $\piu'=(s_{i_2},s_{i_3},\dots,s_{i_\ell})$.
Then
$$\frac{c^h(\X^o_\piu\to \NN)}{e^h(\NN)}=A_{i_1}^h\left(\frac{c^h(\X_{\piu'}^o\to\NN)}{e^h(\NN)}\right)\,.$$
\end{theorem}
Note that our construction is valid also for the permutations $\pi$ which are not given by  Lemma \ref{sequence_lemma}.
\begin{proof}
We use the additivity relation:
\begin{align*}c^h\left((P_{i_1}\times^B\X^o_{\piu'})\setminus g_\piu^{-1}(p_0)\to \NN\right)&=
c^h(P_{i_1}\times^B\X^o_{\piu'}\to\NN)-c^h(g_\piu^{-1}(p_0)\to \NN)\\&=
c^h(P_{i_1}\times^B\X^o_{\piu'}\to\NN)-c^h(\X^o_{\piu'}\to\NN)\,.\end{align*}
We argue as in the proof of Theorem \ref{fund-indukcja}. We note that
$$c^h\left(P_{i_1}\times^B\X^o_{\piu'}\to\NN\right)_q=c^h(\X^o_{\piu'}\to\NN)_q\cdot c^h(T_p\PP^1)\,.$$
The torus action on the first factor has to be twisted by $s_i$ if $p=p_\infty$.
Comparing with the calculation for fundamental classes we multiply the summands by $\frac{c^h(T_p\PP^1)}{e^h(T_p\PP^1)}$, not just by $\frac1{e^h(T_p\PP^1)}$.
\end{proof}
\begin{rmk} The operator $A^h_i$ can also be written as
$$A^h_i(x)=\bett^h_i\left(\frac{c^h(t_{i+1}t_i^{-1})}{e^h(t_{i+1}t_i^{-1})}\,x\right)-x\,.$$
\end{rmk}

For Chern-Schwartz-MacPherson classes we have
\begin{cor}\label{csm-indukcja} Suppose that $\dim (\O_{w})=\dim (\O_{w'})+1$ and $s_i\cdot N_{w'}=N_w$. Then
$$\frac{\csm(\O_{w}\subset\NN)}{e(\NN)}=A_i\left(\frac{\csm(\O_{w'}\subset\NN)}{e(\NN)}\right)\,,$$
where
$$A_i(x)=
\frac 1{t_{i+1}-t_i}x+\frac {1+t_i-t_{i+1}}{t_i-t_{i+1}} s_i(x)\,.$$
Alternatively
$$A_i(x)=
\bett_i
\left(\frac {1+t_{i+1}-t_{i}}{t_{i+1}-t_{i}} x\right)-x=
\bett_i
\left( x\right)+s_i(x)=-\partial_i
\left( x\right)+s_i(x)\,.$$
\end{cor}
For the motivic Chern classes in K-theory
\begin{cor}\label{mc-indukcja} Suppose that $\dim (\O_{w})=\dim (\O_{w'})+1$ and $s_i\cdot N_{w'}=N_w$. Then
$$\frac{\mC(\O_{w}\subset\NN)}{e^K(\NN)}=A_i^K\left(\frac{\mC(\O_{w'}\subset\NN)}{e^K(\NN)}\right)\,,$$
where
$$A^K_i(x)=
\frac {(1+y)t_i/t_{i+1}}{1-t_i/t_{i+1}}x+ \frac {1+y\,t_{i+1}/t_i}{1-t_{i+1}/t_i} s_i(x)\,.$$
Alternatively
$$A^K_i(x)=
\bett^K_i
\left(\frac {1+y\,t_{i}/t_{i+1}}{1-t_{i}/t_{i+1}} x\right)-x\,.$$
\end{cor}

The operator $A_i$ has  appeared in \cite[\S5]{AMSS0} (denoted by $\bf L^\vee$). The K-theoretic version is present in \cite{AMSS} (denoted by $\mathcal T^\vee$). In both cases these operators are dual to those computing the classes of the open orbits.

Similarly as in Proposition \ref{fund-start} we find the beginning of the induction.

\begin{prop}The class of the minimal orbit of rank $m$ divided by the Euler class is equal
$$\frac{[\O_{w_m}]_h}{e^h(\NN)}=\prod_{ n-m+i< j\leq n} \frac{c^h(u\,t_it_j^{-1})}{e^h(u\,t_it_j^{-1})}\cdot \prod_{ n-m+i= j\leq n} \left(\frac{c^h(u\,t_it_j^{-1})}{e^h(u\,t_it_j^{-1})}-1\right)\,.$$
In particular
$$\frac{\csm(\O_{w_m}\subset\NN)}{e(\NN)}=\prod_{ n-m+i<j\leq n} \frac{1+u+t_i-t_j}{u+t_i-t_j}\cdot \prod_{ n-m+i= j\leq n} \frac1{u+t_i-t_j}\,.$$
and
$$\frac{\mC(\O_{w_m}\subset\NN)}{e^K(\NN)}=\prod_{ n-m+i< j\leq n} \frac{1+y\,t_j/(u\, t_i)}{1-t_j/(u\, t_i)}\cdot\prod_{ n-m+i= j\leq n} \frac{(1+y)\,t_j/(u\, t_i)}{1-t_j/(u\, t_i)}\,.$$
\end{prop}

\subsection{Example of computations}
\begin{ex} Suppose $n=4$, $w=\trans 12\trans 34$
$$
\xymatrix@-1pc{
_1\ar@{-}[r]
\ar@{-}@/^0.7pc/[r]
&
_2\ar@{-}[r]
&
_3\ar@{-}[r]\ar@{-}@/^0.7pc/[r]&
_4
}\qquad
X_w=\left\{
\hbox{\footnotesize$\begin{pmatrix}0\ a\ b\ c\\[-0.1cm]0\ 0\ 0\ d\\[-0.1cm]0\ 0\ 0\ e\\[-0.1cm]0\ 0\ 0\ 0\end{pmatrix}$}
~~|~~ ad+be=0\right\}.
$$
Since $N_w=s_2(N_{w_2})$ we have in cohomology
\begin{align*}[X_w]&=-e(\NN)\,\partial_2\left(\frac1{(t_1-t_3+u)(t_1-t_4+u)(t_2-t_4+u)}\right)\\
&=e(\NN)\frac{t_1-t_4+2u}{(t_1-t_2+u)(t_1-t_3+u)(t_1-t_4+u)(t_2-t_4+u)(t_3-t_4+u)}\,.
\end{align*}
\begin{align*}
\csm(\O_w\subset \NN)&=e(\NN)\,A_2\left(\frac{1+t_1-t_4+u}{(t_1-t_3+u)(t_1-t_4+u)(t_2-t_4+u)}\right)\\
&=e(\NN)\frac
{(1+t_1 - t_4 + u)(t_1 - t_4 + 2u+(t_1-t_3+u)(t_2-t_4+u))}
{(t_1-t_2+u)(t_1-t_3+u)(t_1-t_4+u)(t_2-t_4+u)(t_3-t_4+u)}
\,.\end{align*}
In K-theory we have
\begin{align*}[X_w]_K&=e^K(\NN)\,\delta_2\left(\frac1
{(1-\tfrac{t_3}{ut_1})(1-\tfrac{t_4}{ut_1})(1-\tfrac{t_4}{ut_2})}\right)\\
&=e^K(\NN)\frac{1 - \tfrac{t_4}{u^2 t_1}
}{(1-\tfrac{t_2}{ut_1})(1-\tfrac{t_3}{ut_1})(1-\tfrac{t_4}{ut_1})(1-\tfrac{t_4}{ut_2})(1-\tfrac{t_4}{ut_3})}\,.
\end{align*}
\begin{multline*}
\mC(\O_w\subset\NN)=e^K(\NN)\,A_2^K\left(\frac{(1+y)^2\tfrac{t_3}{ut_1}\tfrac{t_4}{ut_2}(1+y\tfrac{t_4}{ut_1})}{(1-\tfrac{t_3}{ut_1})(1-\tfrac{t_4}{ut_1})(1-\tfrac{t_4}{ut_2})}\right)\\
=e^K(\NN)\frac
{ (1+y)^2
\frac{t_4}{u^2\,t_1} \left(1+y\frac{{t_4}}{u\,{t_1} }\right)
   \left(1-\frac{{t_2}}{u\,{t_1} }+\frac{{t_2}}{{t_3}}-\frac{{t_4}}{u\,{t_3}
   }+y
   \big(1-\frac{{t_4}}{ u^2\,{t_1}
  }\big)\right) 
}
{(1-\tfrac{t_2}{ut_1})(1-\tfrac{t_3}{ut_1})(1-\tfrac{t_4}{ut_1})(1-\tfrac{t_4}{ut_2})(1-\tfrac{t_4}{ut_3})}
\,.\end{multline*}
\end{ex}
\begin{ex} We come back to the Example \ref{ex1634}: let $w=\trans 1 6\trans 3 4\in \Inv_7$, $\piu=s_2s_6s_4s_5s_6$.
Then
$$\begin{matrix}
{[X_w]}&=&\hfill(-1)^5{e(\NN)}\cdot\partial_2\partial_6\partial_4\partial_5\partial_6&\hskip-2ex\big(\frac1{(t_1-t_6)(t_1-t_7)(t_2-t_7)}\big)\hfill&\in&H^{16}_\T(\NN)\,,\\
{\csm(X_w\subset \NN)}&=&\hfill {e(\NN)}\cdot A^H_2A^H_6A^H_4A^H_5A^H_6&\hskip-2.5ex\big(\frac{1+t_1-t_7}{(t_1-t_6)(t_1-t_7)(t_2-t_7)}\big)\hfill&\in&H^*_\T(\NN)\,,\\
{[X_w\subset \NN]_K}&=&\hfill{e^K(\NN)}\cdot\beta^K_2\beta^K_6\beta^K_4\beta^K_5\beta^K_6&\hskip-2.5ex \big(\frac1{(1-t_7/t_1)(1-t_6/t_1)(1-t_7/t_2)}\big)\hfill&\in&K_\T(\NN)\,,\\
{\mC(X_w\subset \NN)}&=&\hfill{e^K(\NN)}\cdot A^K_2A^K_6A^K_4A^K_5A^K_6&\hskip-2.5ex\big(\frac{(1+y\,t_7/t_1)((1+y)t_6/t_1)((1+y)+t_7/t_2)}{(1-t_6/t_1)(1-t_7/t_1)(1-t_7/t_2)}\big)\hfill&\in&K_\T(\NN)[y]\,.
\end{matrix}$$
The results are too long to present them in print, although they are easily obtained by computer.
\end{ex}

\section{Relation with classical Schubert calculus}
\def\b{\bullet}
\def\c{\circ}
\def\trzykwd#1#2#3#4#5#6#7#8#9{\Big[\begin{matrix}{#1}\sk {#2}\sk {#3}\\[-0.25cm]{#4}\sk {#5}\sk {#6}\\[-0.25cm]{#7}\sk {#8}\sk {#9}\end{matrix}\Big]}

Amazingly from the  case of square-zero matrices one can deduce formulas for the classes of Schubert varieties in the classical flag variety $\GL_n/\B_n$.
We explain that relation below.

\subsection{From nilpotent orbits to Schubert cells}
We consider equivariant cohomology, but a parallel discussion applies to K-theory.  The construction presented in this section is closely related to \cite{KnuMiShi, Knutson} where an interpretation of the Schubert polynomials is given.
We have the Kirwan surjective map:
\begin{equation}\label{Kirwan}\kappa:H^*_{\T_n\times \T_n}(\Hom(\C^n,\C^n))\longrightarrow H^*_{\T_n}(\GL_n/\B_n)\,.\end{equation}
Classes of the Schubert varieties are images of the classes of $\B_n\times \B_n$ orbits in $\Hom(\C^n,\C^n)$.
Furthermore, let us consider the embedding $\Hom(\C^n,\C^n)\hookrightarrow \Hom(\C^{2n},\C^{2n})$ as the block matrices {\footnotesize $\begin{pmatrix}0\ A\\0\,\ 0\end{pmatrix}$}. This map is $\B_{2n}$-equivariant, given that $\B_{2n}$ acts on $\Hom(\C^{n},\C^{n})$ via the natural surjection $\B_{2n}\twoheadrightarrow \B_n\times\B_n$
{\footnotesize $$\begin{pmatrix}B_1\ C~\\0\,\ B_2\end{pmatrix}\mapsto (B_1,B_2)\,.$$}
The $\B_n\times \B_n$ orbits in $\Hom(\C^n,\C^n)$ are mapped to $\B_{2n}$-orbits in $\Hom(\C^{2n},\C^{2n})$ which are contained in the upper-right corner.
Therefore the case of 2-nilpotent matrices contains all information about Schubert classes. It is remarkable that in addition there are involved ``boundary'' classes corresponding to degenerate matrices $A$ and ``exterior'' classes, which are the classes of orbits not contained in the upper-right corner.
This simple observation can be generalized
\begin{prop}Consider the embedding of $\Hom(\C^k,\C^n)$ to $\Hom(\C^{k+n},\C^{k+n})$ as the upper-right block matrices. Then $\B_k\times\B_n$-orbits in  $\Hom(\C^k,\C^n)$ are exactly the $\B_{k+n}$ orbits contained in the upper-right block.\end{prop}
The recursive formulas of the previous sections may serve as tools to compute fundamental classes of the orbit closures, as well as characteristic classes of orbits themselves.
\medskip

Considering nilpotent orbits has an advantage, which allows us to look at the variables of the double Schubert polynomial as one set, not divided into equivariant parameters and Chern classes of the line bundles. From the point of view of square-zero matrices there is no such distinction and exchanging variables between these groups results transgressing the boundary of the upper-right block, as explained in the introduction.

\subsection{Schubert and Grothendieck polynomials}
\label{Schub_as_sq0}
We consider the upper-triangular square-zero matrices of the size $2n\times 2n$.
In this section we neglect the additional variable $u$, since the extended polynomials with this variable can be recovered by a suitable substitution, provided that we do not apply the operation $\beta_n=-\partial_n$. \medskip

For the $\B_{2n}$-orbits contained in the upper-right block we recover the class in $\Hom(\C^n,\C^n)$ multiplying $\frac{[X_w]}{e(\NN)}$ by the class of the block $$e(\Hom(\C^n,\C^n))=\prod_{i=1}^n\prod_{j=n+1}^{2n}(t_i-t_j)$$ and then changing variables (to agree later with the Schubert polynomials convention)
\begin{equation}t_{1}=x_n\,,\quad t_{2}=x_{n-1}\,,\quad \dots\quad t_{n}=x_1\,.\label{podstawienia1}\end{equation}
\begin{equation}t_{n+1}=y_1\,,\quad t_{n+2}=y_{2}\,,\quad \dots\quad t_{2n}=y_n\,.\label{podstawienia2}\end{equation}
\medskip

Let us recall the inductive definition of the double Schubert polynomials \cite[Def.~6.1]{AndersonFulton}:
\begin{align*}{\rm(i)}\quad&\Sch_{\pi_n}=\prod_{i+j\leq n} (x_i-y_i)\\
{\rm(ii)}\quad&\Sch_{\pi s_i}=\partial^x_i\Sch_\pi\quad\text{ if }\ell(\pi s_i)>\ell(\pi)\\
\end{align*}
Here $\pi_n$ is the longest permutation and the divided difference $\partial^x_i$ is the standard one
$$\partial^x_i(f)=\frac{f-f_{x_i\leftrightarrow x_{i+1}}}{x_i-x_{i+1}}\,.$$
In addition the double Schubert polynomials satisfy
$${\rm(iii)}\quad\Sch_{s_i\pi }=-\partial^y_i\Sch_\pi\quad\text{ if }\ell( s_i\pi)>\ell(\pi)\,,$$
see \cite[Theorem 1.1]{IkMiNa}.
Suppose $w\in \Inv_{2n}$ is an involution such that $N_w$ is contained in the upper right block and the rank of $N_w$ is maximal. Then
the upper right block is a $n\times n$ permutation matrix. Let us reverse the order of rows. The resulting matrix is the matrix of a permutation $\pi_w^{-1}$.

\begin{ex}Let $n=3$,
$w=(1\,4)(2\,6)(3\,5)\in \Inv_6$.
The resulting permutation is obtained by the following sequence of operations:

\def\bskih{\hskip-2pt}
$$\left[\hskip-1pt\begin{matrix}
\c\bskih \c\bskih \c\bskih \b\bskih \c\bskih \c\phantom{.}\\[-0.25cm]
\c\bskih \c\bskih \c\bskih \c\bskih \c\bskih \b\phantom{.}\\[-0.25cm]
\c\bskih \c\bskih \c\bskih \c\bskih \b\bskih \c\phantom{.}\\[-0.25cm]
\c\bskih \c\bskih \c\bskih \c\bskih \c\bskih \c\phantom{.}\\[-0.25cm]
\c\bskih \c\bskih \c\bskih \c\bskih \c\bskih \c\phantom{.}\\[-0.25cm]
\c\bskih \c\bskih \c\bskih \c\bskih \c\bskih \c\phantom{.}\end{matrix}\hskip-6pt\right] ~~~\mapsto~~~\trzykwd \b\c\c \c\c\b \c\b\c
\rotatebox[origin=c]{270}{$\mathlarger{\mathlarger{\curvearrowright\hskip-2.5ex\curvearrowleft}}$}
~~~\mapsto~~~\trzykwd \c\b\c \c\c\b \b\c\c ~~~\mapsto~~~ \big(\,\pi_w(1)=2,~~\pi_w(2)=3,~~\pi_w(3)=1\,\big)=s_1s_2\,.$$

\end{ex}

\begin{prop}\label{prop:szury}
For $X_w$ contained in the upper-right block
\begin{equation}\label{Schubert_eq}e(\Hom(\C^n,\C^n))\cdot \frac{[X_w]}{e(\NN)}=\Sch_{\pi_w}\end{equation}
after the substitution given by (\ref{podstawienia1}--\ref{podstawienia2}).
\end{prop}
This proposition is not new, see \cite[\S2, Prop. 2]{KnutsonZinn:2014}. We include the proof for completeness.

\begin{proof}First let us note that for a polynomial in $t_1,t_2,\dots,t_{2n}$ the operations $\bett_i$ for $i\neq n$ commute with multiplication by $e(\Hom(\C^n,\C^n))$, since that class is symmetric with respect to the groups of variables $\{t_1,t_2,\dots t_n\}$ and $\{t_{n+1},t_{n+2},\dots t_{2n}\}$ separately. Therefore
$$e(\Hom(\C^n,\C^n))\cdot\bett_{i_1}\bett_{i_2}\dots\bett_{i_\ell}\left(\tfrac{[X_{w_n}]}{e(\NN)}\right)=
\bett_{i_1}\bett_{i_2}\dots\bett_{i_\ell}\left(e(\Hom(\C^n,\C^n))\cdot\tfrac{[X_{w_n}]}{e(\NN)}\right)\,.$$
The induction starts with $w_n$
$$e(\Hom(\C^n,\C^n))\cdot\tfrac{[X_{w_n}]}{e(\NN)}=\prod_{i<j}(t_i-t_j)\,.$$
After the change of variables this product is equal to $\Sch_{\pi_n}$, i.e.~the Schubert polynomial associated to the longest permutation. The operators $\bett_i$ for $i\leq n$ become $\partial^x_{n-i}$ (in addition  $\bett_i$ for $i>n$ become $-\partial^y_{i-n}$).
Therefore both sides of the equation \eqref{Schubert_eq}  satisfy the same recursion.
\end{proof}
Applying exactly the same proof we obtain a parallel result for Grothendieck polynomials, confirming the conclusion of the formula in \cite[\S5.4]{Zinn:2018}.
\begin{prop}\label{prop:groth}
For $X_w$ contained in the upper-right block
\begin{equation}\label{Schubert_eq2}e^K(\Hom(\C^n,\C^n))\cdot \frac{[\O_{X_w}]}{e^K(\NN)}={\frak G}_{\pi_w}\end{equation}
after the substitution given by (\ref{podstawienia1}--\ref{podstawienia2}).
\end{prop}
Here the Grothendieck polynomial of the longest permutation is equal to
$$\mathfrak G_{\pi_n}=\prod_{i+j=n}\big(1-\frac{y_j}{x_i}\big)\,.$$
To compute Grothendieck polynomials of smaller length permutation we apply he isobaric divided differences are taken with respect to the variables $x_i$
$$(\partial^{x,K}_i f)(x;y)=\frac {f(\dots,x_i,x_{i+1}\dots;y)}{1-x_{i+1}/x_i}+\frac {f(\dots,x_{i+1},x_{i}\dots;y)}{1-x_i/x_{i+1}}\,,$$
$$\mathfrak G_{\pi s_i}=\partial^{x,K}_i\mathfrak G_{\pi}$$
according to the convention of \cite[\S2.1]{RimanyiSzenes}.
\medskip

Let us come back to the Schubert polynomials.
The equivariant variables in the classical convention are $y_i$'s:
$H^*_{\T_n}(pt)=\Z[y_1,y_2,\dots,y_n]$. The Kirwan map \eqref{Kirwan} is multiplicative but the formulas known from intersection theory of flag varieties do not hold in $H^*_{\T_n\times \T_n}(\Hom(\C^n,\C^n))$.
The boundary terms contribute to the formulas.

\begin{ex} Consider the case $n=2$. Then $\GL_2/\B_2\simeq\PP^1$ and the cohomology ring is isomorphic to
$$\Z[x_1,x_2,y_1,y_2]/(x_1+x_2-y_1-y_2,~~x_1x_2-y_1y_2)\,.$$ The Kirwan map is the quotient map from $\Z[x_1,x_2,y_1,y_2]$.
The class of the zero dimensional cell, the point $p_0$, is equal to $x_1-y_1$.
We have $$[p_0]^2=(y_2-y_1)[p_0]\in H^4_{\T_2}(\GL_2/\B_2)$$
On the level of $\Hom(\C^2,\C^2)$ we have
$[X_{w_2}]=(t_2-t_3)=(x_1-y_1)$ and
\begin{equation}\label{mnoz_w_P1}[X_{w_2}]^2=(x_1-y_1)^2=(y_2-y_1)(x_1-y_1)+(x_1-y_2)(x_1-y_1)=(y_2-y_1)[X_{w_2}]+[X_{(13)}]\end{equation}
since
$$[X_{(13)}]=(t_2-t_3)(t_2-t_4)=(x_1-y_1)(x_1-y_2)\,.$$
Graphically we present the relation \eqref{mnoz_w_P1} as
$$\big[\begin{matrix}\b\sk \c\\[-0.25cm]\c\sk \b\end{matrix}\big]^2=
(y_2-y_1)\big[\begin{matrix}\b\sk \c\\[-0.25cm]\c\sk \b\end{matrix}\big]+\big[\begin{matrix}\b\sk \c\\[-0.25cm]\c\sk \c\end{matrix}\big]\,.
$$
The Kirwan map sends $[X_{w_2}]$ to $[p_0]$ and the boundary class $[X_{(13)}]$ to zero.
\end{ex}

\begin{ex}Let $n=3$.
The corresponding $\B_n$-orbits in $\Hom(\C^3,\C^3)$ have the following classes
$$\begin{matrix}\trzykwd \b\c\c \c\b\c \c\c\b=
(t_3- t_4)  (t_3 - t_5)  (t_2 - t_4),
\hfill&
\trzykwd \b\c\c \c\c\b \c\b\c=
%
(t_3 - t_4) (t_2 - t_4),
\hfill&
\trzykwd \c\b\c \b\c\c \c\c\b=
%
(t_3 - t_5)  (t_3 - t_4),
\hfill\\
\trzykwd \c\b\c \c\c\b \b\c\c=
%
(t_2 + t_3 - t_4 - t_5),
\hfill&
\trzykwd \c\c\b \b\c\c \c\b\c=
t_3 - t_4,
\hfill&
\trzykwd \c\c\b \c\b\c \b\c\c=
%
1\,.\hfill\end{matrix}
$$
After the substitution we obtain the double Schubert polynomials
%
%
$$\begin{matrix}\trzykwd \b\c\c \c\b\c \c\c\b=
(x_1 - y_1)(x_1 - y_2)(x_2 - y_1) \,,
%
\hfill&\trzykwd \b\c\c \c\c\b \c\b\c=
(x_1 - y_1)  (x_2 - y_1)\,,

%
\hfill&\trzykwd \c\b\c \b\c\c \c\c\b=
(x_1 - y_1)  (x_1 - y_1)\,,
%
%
\hfill\\ \trzykwd \c\b\c \c\c\b \b\c\c=
(x_1 + x_2 - y_1 - y_2)\,,
%
%
\hfill&\trzykwd \c\c\b \b\c\c \c\b\c=
x_1 - y_1\,,
%
%
\hfill&\trzykwd \c\c\b \c\b\c \b\c\c=
1\,.
\hfill\end{matrix}
$$
$$\begin{matrix}\trzykwd \b\c\c \c\b\c \c\c\b=\Sch_{321}\,,
%
\hfill&\trzykwd \b\c\c \c\c\b \c\b\c=\Sch_{231}\,,

%
\hfill&\trzykwd \c\b\c \b\c\c \c\c\b=\Sch_{312}\,,
%
%
\hfill& \trzykwd \c\b\c \c\c\b \b\c\c=\Sch_{132}\,,
%
%
\hfill&\trzykwd \c\c\b \b\c\c \c\b\c=\Sch_{213}\,,
%
%
\hfill&\trzykwd \c\c\b \c\b\c \b\c\c=\Sch_{123}\,.
%
\hfill\end{matrix}
$$
We do not list all boundary classes except the following example
$$
\trzykwd \c\b\c \b\c\c \c\c\c
=\bett_4\bett_5\bett_4[X_{w_2}]=\bett_5\bett_4\bett_1[X_{w_2}]=(x_1 -y_1)  (x_1 - y_2)  (x_1 - y_3)\,.$$
The multiplication formulas differ from those for Schubert classes.
We have an equality of polynomials
$$\trzykwd \c\c\b \b\c\c \c\b\c \cdot
\trzykwd \b\c\c \c\c\b \c\b\c =
(y_3-y_1)\trzykwd \b\c\c \c\c\b \c\b\c+
\trzykwd \c\b\c \b\c\c \c\c\c\,.
$$
The class $\trzykwd \c\b\c \b\c\c \c\c\c$ restricts to zero in the cohomology of flag varieties since it corresponds to the orbit consisting of degenerate matrices.
We will not expand the subject of multiplication in the present paper.

\end{ex}

\subsection{Porteous formula} We return to the case of upper-triangular matrices of the size $n\times n$.
The following computation demonstrates a relation between various Schubert classes on the example of rank one matrices. Similar relations hold for higher ranks.\medskip

Suppose $w=\trans i j$. The variety $X_{\trans i j}$ is contained in the upper-right block of the size $i\times (n-j+1)$.
Let $\NN_{i,j}\subset \NN$ be the vector space of the matrices having zeros outside that block. By Proposition \ref{prop:rank_conditions_trans} the variety $X_{\trans i j}\subset \NN_{i,j}$ is identified with the set of matrices of the rank $\leq 1$.
The fundamental class in $\NN_{i,j}$ is described by the Porteous formula, see \cite[Theorem 14.4]{Fulton} or
\cite[12.4]{3264_and_all_that}
\begin{equation} \label{eq:schur_equality}
  [X_{\trans i j}]^\T_{\NN_{i,j}}=\Delta\hbox{$i-1\atop n-j$}(c_\bullet^{[i,j]})=\det\left(\left\{c_{i-1+r-s}^{[i,j]}\right\}_{1\leq r,s\leq n-j}\right)
\,.
\end{equation}
where \(c_\bullet^{[i,j]}\)
is given by \begin{equation}\label{chern_sequence}
 c_\bullet^{[i,j]}=\frac{\prod_{r=1}^i(1+t_r)}{\prod_{s=j}^n(1+t_s)}=1+c_1^{[i,j]}+c_2^{[i,j]}+\dots\,.
\end{equation}

\begin{ex}
  Consider \(n=8\) and \(w=\trans{4}{6} \in \Inv_8\). Then by the
  considerations  above we obtain
  \[[X_w]^\T_{\NN_{4,6}}=
\Delta\hbox{$4-1\atop 8-6 $}(c^{[4,6]}_\bullet)
=\det \begin{pmatrix} c_3 & c_4 \\
      c_2 & c_3 \end{pmatrix}=c_3^2-c_2c_4\,.\]
\end{ex}
To get rid of the dependence on the ambient space we divide by the Euler class
 \begin{equation}\label{eq:class}
\frac{[X_{\trans i j}]^\T}{e(\NN)}
=
\frac{\Delta\hbox{$i-1\atop n-j$}(c_\bullet^{[i,j]})}{e(\NN_{i,j})}\,,
\end{equation}
where the Euler class of $\NN_{i,j}$ is given
by the formula
$$e(\NN_{i,j})=\prod_{r=1}^i\prod_{s=j}^n(t_r-t_s)\,.$$
By Theorem (\ref{th:fund-indukcja}) applying $\bett_i$ to the class \eqref{eq:class} we obtain a relation
\begin{prop} Suppose $i+1<j$. Then
  \begin{equation*}
      \frac{\Delta^i_{n-j}
(c_\bullet^{[i+1,j]})}{e(\NN_{i+1,j})} = \bett_i \left( \frac{\Delta\hbox{$i-1\atop n-j$}(c_\bullet^{[i,j]})}{e(\NN_{i,j})} \right)
  \end{equation*}
\end{prop}
An analogous relation we obtain when we switch the roles of rows and columns.
The above procedure may serve to compute $\Delta^{i-1}_{n-j}(c^{[i,j]}_\bullet)$ using various path
recursion, not having to compute virtual Chern classes.  For example, if \(n=8\) and
\(w=\trans{4}{5}\), then the following two walks from the north-east corner to (4,5)

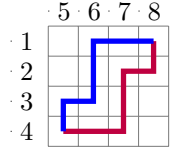
\begin{center}
\begin{tikzpicture}
  \draw[step=0.4cm, color=gray] (0, 0) grid (1.6,1.6);
  \draw[line width=2pt, purple] (0.2,0.2) -- (1,0.2)  -- (1,1) -- (1.4,1) -- (1.4,1.4);
  \draw[line width=2pt, blue] (0.2,0.2) -- (0.2,0.6) -- (0.6,0.6) -- (0.6,1.4)  -- (1.4,1.4);
\draw(-0.5,1.4) -- (-0.5,1.4) node [right,text centered,midway]{$1$};
\draw(-0.5,1) -- (-0.5,1) node [right,text centered,midway]{$2$};
\draw(-0.5,0.6) -- (-0.5,0.6) node [right,text centered,midway]{$3$};
\draw(-0.5,0.2) -- (-0.5,0.2) node [right,text centered,midway]{$4$};
\draw(0,1.8) -- (0,1.8) node [right,text centered,midway]{$5$};
\draw(0.4,1.8) -- (0.4,1.8) node [right,text centered,midway]{$6$};
\draw(0.8,1.8) -- (0.8,1.8) node [right,text centered,midway]{$7$};
\draw(1.2,1.8) -- (1.2,1.8) node [right,text centered,midway]{$8$};
\end{tikzpicture}
\end{center}
correspond to reduced word expressions {\color{blue}\(\bett_3\bett_5\bett_2\bett_1\bett_6\bett_7\)} and
{\color{purple}\(\bett_5\bett_6\bett_3\bett_2\bett_7\bett_1\)}.

\section{Relation with RTV weight function}

In a series of papers, see e.g. \cite{RTV0,RTV} Rim{\'a}nyi, Tarasov and Varchenko have studied certain rational functions, called \emph{trigonometric weight functions}. Their homological analogue were described  in \cite{RV}. The weight functions are the limits of elliptic weight functions which are rational combinations of Jacobi theta functions. The trigonometric weight functions define elements in K-theory of flag varieties, which satisfy axioms of the K-theoretic stable envelopes in the sense of Okounkov, \cite{Oko}. The trigonometric functions depend on a choice of the Schubert cell and the slope. For a preferred choice of the slope the trigonometric weight function  defines (by certain substitutions of parameters) an element in the K-theory of the flag variety. This element is equal to the motivic Chern class of the Schubert cell. This was proven in \cite{FRW} by verification of the axioms of stable envelopes.  The weight function itself remained mysterious.
\medskip

Let us discuss the case of the full flag varieties $\Fell(n)=\GL_n/B_n$. We will not give here all the technical definitions. We just want to point out how one can interpret the recursion obtained in Theorem \ref{char-indukcja}.
For our purposes it is the most convenient to consider the modified version of the weight function given in \cite{FRW2},
or the limit of the
modified elliptic weight function $\bf \hat w_\omega$ from  \cite[\S6]{RW}. The function considered by us
depends on two sets of variables $\{\gamma_i\}_{1\leq i\leq n-1}$ and $\{z_i\}_{1\leq i\leq n}$.
It is obtained from the weight function of \cite[\S5.4]{FRW2} by the substitution
$\alpha_i^{(k)}=\gamma_i$ for $i\leq k<n$.
The trigonometric weight function satisfies the R-matrix relations, which result in the recursion
\cite[Theorem 3.2, (3.13)]{RTV} (for the distinguished slope $\Delta=\{m_{k,l}\equiv -1\}$)
\begin{equation}
 W_{s_{a}\tau}^\Delta(\pmb{\alpha},{\bf z})\,=\,
\frac{1-h^{-1}z_{a+1}/z_a}{1-z_{a+1}/z_a}\,
 W_\tau^\Delta(\pmb\alpha,s_{a}({\bf z}))+
(h^{-1}-1)\,\frac{1}{1-z_{a+1}/z_a}\,W_\tau^\Delta(\pmb\alpha,{\bf z})
\end{equation}
if $\ell(s_i\tau)>\ell(\tau)$. After substitution $h=-y^{-1}$ and reorganizing the second summand we obtain
\begin{equation}\label{rtv-rec}
\widetilde W_{s_{a}\tau}(\pmb\gamma,{\bf z})\,=\,
\frac{1+yz_{a+1}/z_a}{1-z_{a+1}/z_a}\,
\widetilde W_\tau(\pmb\gamma,s_{a}({\bf z}))+
(1+y)\,\frac{z_{a}/z_{a+1}}{1-z_{a}/z_{a+1}}\,\widetilde W_\tau(\pmb\gamma,{\bf z})\,.
\end{equation}
This is exactly the recursion of Corollary \ref{mc-indukcja} with the variables $t_i$ replaced by $z_i$.
Our goal in this section is to give a topological interpretation of the trigonometric weight function.
\medskip

We present the flag variety $\Fell(n)$ as the quotient of the Stiefel variety $$\text{Stief}(n-1,n)=\{\varphi\in \Hom(\C^{n-1},\C^n)~:~\varphi\text{ is injective}\}\,, $$
$$p:\text{Stief}(n-1,n)\to \Fell(n)=\text{Stief}(n-1,n)/\B_{n-1}\,.$$
The resulting (Kirwan) map
$$\kappa:K_{\T_{n-1}\times \T_n}(\Hom(\C^{n-1},\C^n))\to K_{\T_n}(\Fell(n))$$
is surjective.
Let $\gamma_1,\gamma_2,\dots,\gamma_{n-1}$ be the equivariant variables of $\T_{n-1}$ and let $z_1,z_2,\dots,z_{n}$ be the equivariant variables of $\T_n$.
According to \cite[\S8]{FRW2} applied to the quotient $\text{Stief}(n-1,n)/\B_{n-1}$ for a $\B_{n-1}\times \B_n$ orbit $\O\subset \text{Stief}(n-1,n)$ we have
$$\mC(p(\O)\subset \Fell(n))=\kappa\big(\mC(\O\subset \Hom(\C^{n-1},\C^n))/\lambda_y(\mathfrak b_{n-1})\big)\,,$$
where $$\lambda_y(\mathfrak b_{n-1})=\prod_{i\leq j}\left(1+y\tfrac{\gamma_j}{\gamma_i}\right)\,.$$
Let us identify the space $\Hom(\C^{n-1},\C^n)$ with the subspace of $\NN_{n,n-1}\subset \Hom(\C^{2n-1},\C^{2n-1})$ consisting of the a upper-right block matrices.
We set
\begin{equation}t_{1}=z_1\,,\quad t_{2}=z_{2}\,,\quad \dots\quad t_{n}=z_n\,,\label{ztsubs}\end{equation}
\begin{equation}t_{n+1}=\gamma_1\,,\quad t_{n+2}=\gamma_{2}\,,\quad \dots\quad t_{2n-1}=\gamma_{n-1}\,,\end{equation}
and
\begin{equation}y=-h^{-1}\,.\label{yhsubs}\end{equation}
(The last substitution is already present in \cite{FRW,FRW2}.)
As in \S\ref{Schub_as_sq0} the Borel orbit classes $\mC(\O_w\subset \NN)/e^K(\NN)$ satisfy the recursion of Corollary \ref{mc-indukcja}. If we restrict our attention to the recursion not involving $\gamma$-variables, i.e. $A^K_i$ for $i<n$, we can multiply by $e^K(\Hom(\C^{n-1},\C^n))$. Hence
$$\mC(\O_{s_iw}\subset \Hom(\C^{n-1},\C^n))=A^K_i\left(\mC(\O_w\subset \Hom(\C^{n-1},\C^n)\right)\,.
$$
This is exactly the recursion \eqref{rtv-rec}.
It remains to compare the motivic Chern class with the weight function for the minimal orbit:
$$\mC\big(\O_{w_{n-1}}\subset \Hom(\C^{n-1},\C^n)\big)=\prod_{1=j}^{n-1} \left(\prod_{i=1}^{j-1} (1+y \tfrac{\gamma_j}{z_i})\cdot (1+y) \tfrac{\gamma_j}{z_j}\cdot \prod_{i=j+1}^n(1- \tfrac{\gamma_j}{z_i})\right)\,.
$$
We leave as an exercise to specialize the definition of the weight function (e.g the one given in \cite[\S5.4]{FRW2}) and see that
$$\widetilde W_{id}= \frac{\mC(\O_{w_{n-1}}\subset \Hom(\C^{n-1},\C^n))}{\prod_{i\leq j}\left(1+y\tfrac{\gamma_j}{\gamma_i}\right)}\,.$$
By the recursion it follows that
\begin{cor}Let $\tau \in\Sn$ be a permutation and $w_\tau$ the nilpotent matrix given by the map defined on the standard basis $\{\varepsilon_i\}_{i\leq 2n-1}$
$$\varepsilon_i\mapsto 0\quad\text{ for }1\leq i\leq n,\qquad \varepsilon_{n+j}\mapsto \varepsilon_{\tau(j)}\quad\text{ for }1\leq j\leq n-1\,.$$
Then
$$\widetilde W_{\tau}= \frac{\mC(\O_{w_\tau}\subset \Hom(\C^{n-1},\C^n))}{\lambda_y(\mathfrak b_{n-1})}$$
after the substitutions (\ref{ztsubs}--\ref{yhsubs}).
\end{cor}
This shows that (up to the factor corresponding to the group $\B_{n-1}$ by which we divide) the modified weight function is the motivic Chern class of the nilpotent $\B_{2n-1}$-orbit.
\medskip

\begin{ex} Following the definition of  \cite[\S5.4]{FRW2} we compute the trigonometric weight function for the flag variety $\Fell(2)$:
$$\widetilde W_{12}=\frac{\gamma_1}{z_1}\cdot (1-\frac{\gamma_1}{z_2}),\qquad
\widetilde W_{21}=(1+y\frac{\gamma_1}{z_2})\cdot\frac{\gamma_1}{z_2}\,.$$
 The trivial permutation corresponds to the minimal orbit in $\NN_3$ of rank 1:
$$
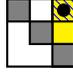
\begin{tikzpicture}[thick,scale=0.3]
\draw[fill=yellow](2,1) rectangle (3,3);
\foreach \x in {0,1,2}
  \draw[fill=gray] (\x ,3-\x ) rectangle (\x+1 ,2-\x );
\draw[pattern=north west lines, pattern color=blue] (2,2) rectangle (3,3);
\draw(0,0) rectangle (3,3);
\draw[fill=black] (2.5,2.5) circle (0.25);
\end{tikzpicture}
$$
After the substitution  $(z_1,z_2,\gamma_1)=(t_1,t_2,t_3)$ the above class is equal to $$(y+1)^{-1}\mC(\O_{(3,1)}
\subset \Hom(\C,\C^2))\,.$$
Here $\C$ is given the weight $\gamma_1=t_3$ and $\C^2$ the weights $z_1=t_1$ and $z_2=t_2$.
 Note that $(1+y)=\lambda_y({\mathfrak b_1})$.\medskip

  For the open  orbit in $\Hom(\C^1,\C^2)\subset \NN_3$
$$
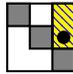
\begin{tikzpicture}[thick,scale=0.3]
\draw[fill=yellow](2,1) rectangle (3,3);
\foreach \x in {0,1,2}
  \draw[fill=gray] (\x ,3-\x ) rectangle (\x+1 ,2-\x );
\draw[pattern=north west lines, pattern color=blue] (2,1) rectangle (3,3);
\draw(0,0) rectangle (3,3);
\draw[fill=black] (2.5,1.5) circle (0.25);
\end{tikzpicture}
$$
after the substitution  $h=-y$ the weight function $\widetilde W_{21}$ is equal to $$(y+1)^{-1}\mC(\O_{(2,3)}
\subset \Hom(\C,\C^2))\,.$$

\end{ex}

\begin{ex} $$\begin{tikzpicture}[thick,scale=0.2]
\draw[fill=yellow](3,2) rectangle (5,5);
\foreach \x in {0,1,2,3,4}
  \draw[fill=gray] (\x ,5-\x ) rectangle (\x+1 ,4-\x );
\draw[pattern=north west lines, pattern color=blue] (3,4)--(4,4)--(4,3)--(5,3)--(5,5)--(3,5)--(3,4);
\draw(0,0) rectangle (5,5);
\draw[fill=black]  (3.5,4.5) circle (0.25);
\draw[fill=black] (4.5,3.5) circle (0.25);
\end{tikzpicture}
$$
The  0-dimensional cell in $\Fell(3)$ corresponds to the minimal rank 2 orbit in $\NN_5$.
Following  the definition of \cite[\S5.4]{FRW2} we find that the weight function is equal to
$$\widetilde W_{123}= (1 +y\tfrac{\gamma_2 }{\gamma_1})^{-1}\tfrac{\gamma_1 \gamma_2}{z_1 z_2} (1 + y\tfrac{\gamma_2 }{z_1}) (1 - \tfrac{\gamma_1}{z_2}) (1 - \tfrac{\gamma_1}{z_3}) (1 - \tfrac{\gamma_2}{z_3})$$
This  function can be written as
$$\big((y+1)^2(1 +y\tfrac{\gamma_2 }{\gamma_1}) \big)^{-1}\cdot (1+y)\tfrac{\gamma_1}{z_1}  (1 +y \tfrac{\gamma_2 }{z_1})\cdot (1 - \tfrac{\gamma_1}{z_2}) (1+y)\tfrac{\gamma_2} {z_2} \cdot (1 - \tfrac{\gamma_1}{z_3}) (1 - \tfrac{\gamma_2}{z_3})$$
Which is clearly  equal to $\lambda_{y}({\mathfrak b}_2)^{-1}\cdot\mC(\O_{w_2}\subset \Hom(\C^2,\C^3))$
\end{ex}

\begin{rmk}The original weight function depends on the parameters
$$
\big\{\alpha^{(k)}_i\big\}_{i\leq k< n}\quad\text{and}\quad \big\{z_i\big\}_{i\leq n} \,,$$
while the modified weight function is obtained by the substitution $\alpha^{(k)}_i=\gamma_i$ for $i\leq k<n$ and by the division by $\lambda_{-h}(\bigoplus_{k=1}^{n-1} \mathfrak{gl}_k)$.
 Presumably the original weight function is related to the motivic Chern classes of orbits of $$G=\prod_{k=1}^{n-1}\GL_k\times \B_n$$ in the representation space $$V=\bigoplus_{k=2}^n\Hom(\C^{k-1},\C^k)\,.$$
It is desirable to write an explicit resolutions of orbit closures and write formulas for characteristic classes of all orbits of $G$ in $V$, not only those descending to $\Fell(n)$. We leave this task for future.\end{rmk}


\end{document}